\documentclass[reqno]{amsart}
\usepackage{hyperref}

\AtBeginDocument{{\noindent\small
\emph{Electronic Journal of Differential Equations},
Vol. 2018 (2018), No. 29, pp. 1--24.\newline
ISSN: 1072-6691. URL: http://ejde.math.txstate.edu or http://ejde.math.unt.edu}
\thanks{\copyright 2018 Texas State University.}
\vspace{8mm}}

\begin{document}
\title[\hfilneg EJDE-2018/29\hfil
   properties of fractional differentiation operators]
{Spectral properties of fractional differentiation operators}

\author[M. V. Kukushkin
  \hfil EJDE-2018/29\hfilneg]
{Maksim V. Kukushkin}

\address{Maksim V. Kukushkin \newline
International Committee "Continental",
Geleznovodsk 357401, Russia}
\email{kukushkinmv@rambler.ru}

\dedicatory{Communicated by Ludmila S. Pulkina}

\thanks{Submitted October 10, 2017. Published January 29, 2018.}
\keywords{Fractional derivative; fractional integral;    energetic space;
\hfill\break\indent sectorial operator;
strong accretive operator;}

\begin{abstract}
We consider   fractional differentiation operators in various senses and  show that the strictly accretive property
   is the common property  of fractional differentiation operators. Also we prove that the sectorial property holds for differential operators  second order with a fractional derivative in the  final term, we explore a location of the spectrum and  resolvent sets   and show that the  spectrum  is  discrete. We prove that there exists  a two-sided estimate for eigenvalues of the real component of operators second order with the fractional derivative in  the final term.
\end{abstract}

\maketitle
\numberwithin{equation}{section}
\newtheorem{theorem}{Theorem}[section]
\newtheorem{lemma}[theorem]{Lemma}
\newtheorem{remark}[theorem]{Remark}
\allowdisplaybreaks

\section{Introduction}
It is remarkable that the term accretive,  which applicable to a linear operator $T$ acting in   Hilbert space $H,$ is introduced by Friedrichs in the paper \cite{firstab_lit:fridrichs1958}, and means that the operator $T$ has the following property:   the numerical range $\Theta(T)$ (see \cite[p.335]{firstab_lit:kato1966}) is a subset of the right half-plane i.e.
    $$
    {\rm Re}\left( Tu,u\right)_{H}\geq0,\;u\in \mathfrak{D}(T).
    $$
Accepting  the   notation  of the paper  \cite{firstab_lit:kipriyanov1960} we assume that $\Omega$ is a  convex domain of the  $n$ -  dimensional Euclidean space $\mathbb{E}^{n}$, $P$ is a fixed point of the boundary $\partial\Omega,$
$Q(r,\mathbf{e})$ is an arbitrary point of $\Omega;$ we denote by $\mathbf{e}$   a unit vector having the direction from  $P$ to $Q,$ denote by $r=|P-Q|$   an Euclidean distance between the points $P$ and $Q.$ We use the shorthand notation    $T:=P+\mathbf{e}t,\,t\in \mathbb{R}.$
We   consider the Lebesgue  classes   $L_{p}(\Omega),\;1\leq p<\infty $ of  complex valued functions.  For the function $f\in L_{p}(\Omega),$    we have
\begin{equation}\label{1}
\int\limits_{\Omega}|f(Q)|^{p}dQ=\int\limits_{\omega}d\chi\int\limits_{0}^{d(\mathbf{e})}|f(Q)|^{p}r^{n-1}dr<\infty,
\end{equation}
where $d\chi$   is the element of the solid angle of
the unit sphere  surface (the unit sphere belongs to $\mathbb{E}^{n}$)  and $\omega$ is a  surface of this sphere,   $d:=d(\mathbf{e})$  is the length of the  segment of the  ray going from the point $P$ in the direction
$\mathbf{e}$ within the domain $\Omega.$
Without lose of   generality, we consider only those directions of $\mathbf{e}$ for which the inner integral on the right side of equality \eqref{1} exists and is finite. It is  the well-known fact that  these are almost all directions.
We denote by   ${\rm Lip}\, \mu,\;(0<\mu\leq1) $  the set of functions satisfying the Holder-Lipschitz condition
$$
{\rm Lip}\, \lambda:=\left\{\rho(Q):\;|\rho(Q)-\rho(P)|\leq M r^{\lambda},\,P,Q\in \bar{\Omega}\right\}.
$$
Consider the  Kipriyanov  fractional differential operator     defined in  the paper \cite{firstab_lit:1kipriyanov1960}  by  the formal expression
\begin{equation*}
\mathfrak{D}^{\alpha}(Q)=\frac{\alpha}{\Gamma(1-\alpha)}\int\limits_{0}^{r} \frac{[f(Q)-f(T)]}{(r - t)^{\alpha+1}} \left(\frac{t}{r} \right) ^{n-1} dt+
C^{(\alpha)}_{n} f(Q) r ^{ -\alpha}\!,\, P\in\partial\Omega,
\end{equation*}
where
$
C^{(\alpha)}_{n} = (n-1)!/\Gamma(n-\alpha).
$
In accordance with Theorem 2   \cite{firstab_lit:1kipriyanov1960}, under the assumptions
\begin{equation}\label{2}
 lp\leq n,\;0<\alpha<l- \frac{n}{p} +\frac{n}{q}, \,q>p,
\end{equation}
we have that
     for sufficiently small $\delta>0$ the following inequality holds
\begin{equation}\label{3}
\|\mathfrak{D}^{\alpha}f\|_{L_{q}(\Omega)}\leq \frac{K}{\delta^{\nu}}\|f\|_{L_{p}(\Omega)}+\delta^{1-\nu}\|f\|_{L^{l}_{p}(\Omega)},\, f\in\stackrel{\circ}{W_p ^l} (\Omega),
\end{equation}
where
\begin{equation}\label{4}
 \nu=\frac{n}{l}\left(\frac{1}{p}-\frac{1}{q} \right)+\frac{\alpha+\beta}{l}.
\end{equation}
The constant  $K$ does not depend on $\delta,\,f;$   the point $P\in\partial\Omega ;\;\beta$ is an arbitrarily small fixed positive number.
Further, we assume that $\alpha \in (0,1).$
Using the notation   of the paper  \cite{firstab_lit:samko1987}, we   denote  by $I_{a+}^{\alpha}(L_{p} ),\;I_{b-}^{\alpha}(L_{p} ),\;1\leq p\leq\infty$ the left-side, right-side  classes of functions representable by the fractional integral on the segment $[a,b]$ respectively.  Let $\mathfrak{d}:={\rm diam}\,\Omega ;\;C,C_{i}={\rm const},\,i\in \mathbb{N}_{0}.$ We use a shorthand  notation  $P\cdot Q=P^{i}Q_{i}=\sum^{n}_{i=1}P_{i}Q_{i}$ for the inner product of the points $P=(P_{1},P_{2},...,P_{n}),\,Q=(Q_{1},Q_{2},...,Q_{n})$ which     belong to  $\mathbb{E}^{n}.$
     Denote by  $D_{i}u$  the week  derivative of the function $u$ with respect to a coordinate variable with index   $1\leq i\leq n.$
We  assume that all functions have  a zero extension outside  of $\bar{\Omega}.$  Denote by  $  \mathrm{D} (L), \mathrm{R} (L)$    the domain of definition, range of values of the operator $L$ respectively.
Everywhere further,   unless  otherwise  stated,  we   use the notations of the papers    \cite{firstab_lit:kipriyanov1960}, \cite{firstab_lit:1kipriyanov1960},
\cite{firstab_lit:samko1987}.
Let us  define the operators
$$
 (\mathfrak{I}^{\alpha}_{0+}g)(Q  ):=\frac{1}{\Gamma(\alpha)} \int\limits^{r}_{0}\frac{g (T)}{( r-t)^{1-\alpha}}\left(\frac{t}{r}\right)^{n-1}dt,\,(\mathfrak{I}^{\alpha}_{d-}g)(Q  ):=\frac{1}{\Gamma(\alpha)} \int\limits_{r}^{d }\frac{g (T)}{(t-r)^{1-\alpha}}dt,
$$
$$
\;g\in L_{p}(\Omega),\;1\leq p\leq\infty.
$$
These  operators we call respectively   the left-side, right-side directional fractional integral.
We introduce   the classes of functions representable by the directional fractional integrals.
 \begin{equation}\label{5}
  \mathfrak{I}^{\alpha}_{0+}(L_{p}  ):=\left\{ u:\,u(Q)=(\mathfrak{I}^{\alpha}_{0+}g)(Q  ),\, g\in L_{p}(\Omega),\,1\leq p\leq\infty \right\},
 \end{equation}
\begin{equation}\label{6}
  \mathfrak{I }  ^{\alpha}_{ d   -} (L_{p}  ) =\left\{ u:\,u(Q)=(\mathfrak{I}^{\alpha}_{d-}g)(Q  ),\;g\in L_{p}(\Omega),\;1\leq p\leq\infty  \right\}.
 \end{equation}
Define the operators  $\psi^{+}_{\varepsilon },\; \psi^{-}_{\varepsilon }$ depended on the parameter $\varepsilon>0.$  In the left-side case
 \begin{equation}\label{7}
(\psi^{+}_{  \varepsilon }f)(Q)=  \left\{ \begin{aligned}
 \int\limits_{0}^{r-\varepsilon }\frac{ f (Q)r^{n-1}- f(T)t^{n-1}}{(  r-t)^{\alpha +1}r^{n-1}}  dt,\;\varepsilon\leq r\leq d  ,\\
   \frac{ f(Q)}{\alpha} \left(\frac{1}{\varepsilon^{\alpha}}-\frac{1}{ r ^{\alpha} }    \right),\;\;\;\;\;\;\;\;\;\;\;\;\;\;\; 0\leq r <\varepsilon .\\
\end{aligned}
 \right.
\end{equation}
In the right-side case
\begin{equation*}
 (\psi^{-}_{  \varepsilon }f)(Q)=  \left\{ \begin{aligned}
 \int\limits_{r+\varepsilon }^{d }\frac{ f (Q)- f(T)}{( t-r)^{\alpha +1}} dt,\;0\leq r\leq d -\varepsilon,\\
   \frac{ f(Q)}{\alpha} \left(\frac{1}{\varepsilon^{\alpha}}-\frac{1}{(d -r)^{\alpha} }    \right),\;\;\;d -\varepsilon <r \leq d ,\\
\end{aligned}
 \right.
 \end{equation*}
 where $\mathrm{D}(\psi^{+}_{  \varepsilon }),\mathrm{D}(\psi^{-}_{  \varepsilon })\subset L_{p}(\Omega).$
Using the definitions of the monograph  \cite[p.181]{firstab_lit:samko1987}  we consider the following operators.  In the left-side case
 \begin{equation}\label{8}
 ( \mathfrak{D} ^{\alpha}_{0+\!,\,\varepsilon}f)(Q)=\frac{1}{\Gamma(1-\alpha)}f(Q) r ^{-\alpha}+\frac{\alpha}{\Gamma(1-\alpha)}(\psi^{+}_{  \varepsilon }f)(Q).
 \end{equation}
 In the right-side case
 \begin{equation*}
 ( \mathfrak{D }^{\alpha}_{d-\!,\,\varepsilon}f)(Q)=\frac{1}{\Gamma(1-\alpha)}f(Q)(d-r)^{-\alpha}+\frac{\alpha}{\Gamma(1-\alpha)}(\psi^{-}_{  \varepsilon }f)(Q).
 \end{equation*}
 The left-side  and  right-side fractional derivatives  are  understood  respectively  as the  following  limits with respect to the norm  $L_{p}(\Omega),\,(1\leq p<\infty)$
 \begin{equation}\label{8.1}
 \mathfrak{D }^{\alpha}_{0+}f=\lim\limits_{\stackrel{\varepsilon\rightarrow 0}{ (L_{p}) }} \mathfrak{D }^{\alpha}_{0+\!,\,\varepsilon} f  ,\; \mathfrak{D }^{\alpha}_{d-}f=\lim\limits_{\stackrel{\varepsilon\rightarrow 0}{ (L_{p}) }} \mathfrak{D }^{\alpha}_{d-\!,\,\varepsilon} f .
\end{equation}
 We   need   auxiliary propositions, which  are  presented  in the next section.

 \section{     Some     lemmas and theorems}

We have the following theorem on boundedness of the directional fractional integral operators.
 \begin{theorem}\label{T1} The directional fractional integral operators are bounded in $L_{p}(\Omega),$ $1\leq p<\infty,$ the following estimates holds
  \begin{equation}\label{9}
 \| \mathfrak{I}^{\alpha}_{0 +}u\|_{L_{p}(\Omega)}\leq C\|u \|_{L_{p}(\Omega)},\;\|   \mathfrak{I} ^{\alpha}_{d -}u\|_{L_{p}(\Omega)}\leq C\|u \|_{L_{p}(\Omega)},\;C=  \mathfrak{d} ^{\alpha}/ \Gamma(\alpha+1)  .
 \end{equation}
 \end{theorem}
 \begin{proof}
Let us prove   first estimate  \eqref{9}, the proof of the second one   is absolutely analogous. Using the generalized Minkowski inequality, we have
$$
 \| \mathfrak{I}^{\alpha}_{0 +}u\|_{L_{p}(\Omega)}=\frac{1}{\Gamma(\alpha)}\left( \int\limits_{\Omega}\left|  \int\limits^{r}_{0}\frac{g (T)}{( r-t)^{1-\alpha}}\left(\frac{t}{r}\right)^{n-1}\!\!\!dt  \right|^{p}dQ\right)^{1/p} 
$$
$$
 =\frac{1}{\Gamma(\alpha)}\left( \int\limits_{\Omega}\left|  \int\limits^{r}_{0}\frac{g (Q-\tau  \mathbf{e})}{\tau^{1-\alpha}}\left(\frac{r-\tau}{r}\right)^{n-1}\!\!\!d\tau  \right|^{p}dQ\right)^{1/p} 
$$
$$
  \leq\frac{1}{\Gamma(\alpha)}\left( \int\limits_{\Omega}\left(  \int\limits^{\mathfrak{d}}_{0}\frac{|g (Q-\tau  \mathbf{e})|}{\tau^{1-\alpha}}d\tau  \right)^{p}dQ\right)^{1/p} 
  $$
  $$
\leq \frac{1}{\Gamma(\alpha)}  \int\limits^{\mathfrak{d}}_{0}\tau^{\alpha-1} d\tau \left( \int\limits_{\Omega}  |g (Q-\tau  \mathbf{e})|^{p} dQ  \right)^{1/p}\!\!\leq
 \frac{ \mathfrak{d} ^{\alpha}}{\Gamma(\alpha+1)}\,  \| u\|_{L_{p}(\Omega)}.
$$\end{proof}

\begin{theorem}\label{T2}
Suppose $f\in L_{p}(\Omega),$  there  exists   $\lim\limits_{\varepsilon\rightarrow  0}\psi^{+}_{  \varepsilon }f$ or $\lim\limits_{\varepsilon\rightarrow  0}\psi^{-}_{  \varepsilon }f$ with respect to the norm $L_{p}(\Omega),\,(1\leq p<\infty);$ then   $f\in \mathfrak{I} ^{\alpha}_{0 +}(L_{p}) $ or $f\in \mathfrak{I }^{\alpha}_{d -}(L_{p})$ respectively.
\end{theorem}
\begin{proof}
 Let $f\in L_{p}(\Omega)$ and $\lim\limits_{\stackrel{\varepsilon\rightarrow 0}{ (L_{p}) }}\psi^{+}_{  \varepsilon }f=\psi.$ Consider the function
$$
(\varphi^{+}_{ \varepsilon}f)(Q)=\frac{1}{\Gamma(1-\alpha)}\left\{\frac{f(Q)}{ r ^{\alpha}}+\alpha (\psi^{+}_{ \varepsilon }f)(Q) \right\}.
 $$
 Taking into account     \eqref{7}, we can easily prove  that $\varphi^{+}_{  \varepsilon }f\in L_{p}(\Omega).$ Obviously,    there  exists the limit
 $\varphi^{+}_{  \varepsilon }f\rightarrow \varphi\in L_{p}(\Omega),\,\varepsilon\downarrow 0.$ Taking into account Theorem \ref{T1}, we can  complete the proof,  if we  show that
 \begin{equation}\label{10.01}
 \mathfrak{I}^{\alpha}_{0+}\varphi^{+}_{ \varepsilon }f  \stackrel{L_{p}}{\rightarrow} f,\,\varepsilon\downarrow0.
 \end{equation}
 In the case  $(\varepsilon\leq r\leq d),$ we have
$$
(\mathfrak{I}^{\alpha}_{0 +}\varphi^{+}_{ \varepsilon }f)(Q)\frac{\pi r^{n-1}}{\sin\alpha\pi}=
 \int\limits_{\varepsilon}^{r}\frac{f (P+y\mathbf{e})y ^{n-1-\alpha}}{( r-y)^{1-\alpha} } dy 
$$
$$
  +\alpha\int\limits_{\varepsilon}^{r}( r-y)^{\alpha-1}  dy\int\limits_{0 }^{y-\varepsilon }\frac{ f (P+y\mathbf{e})y^{n-1}- f(T)t^{n-1}}{( y-t )^{\alpha +1}} dt 
$$
$$
 +\frac{1}{\varepsilon^{\alpha}}\int\limits_{0}^{\varepsilon  }f (P+y\mathbf{e})( r-y)^{\alpha-1} y^{n-1}  dy =I.
$$
By direct calculation, we obtain
 \begin{equation}\label{10}
I=   \frac{1}{\varepsilon^{\alpha}}\int\limits_{0}^{r  }f (P+y\mathbf{e})( r-y)^{\alpha-1}y^{n-1}   dy  -
 \alpha\int\limits_{\varepsilon}^{r}( r-y)^{\alpha-1} dy\int\limits_{0 }^{y-\varepsilon }\frac{  f(T)}{(  y-t)^{\alpha +1}}t^{n-1} dt .
\end{equation}
 Changing the variable of integration in  the second  integral,   we have
$$
 \alpha\int\limits_{\varepsilon}^{r}( r-y)^{\alpha-1} dy\int\limits_{0 }^{y-\varepsilon }\frac{  f(T)}{(  y-t)^{\alpha +1}}t^{n-1} dt 
$$
$$
 =\alpha\int\limits_{0}^{r-\varepsilon}( r-y-\varepsilon)^{\alpha-1} dy\int\limits_{0 }^{y  }\frac{  f(T)}{(  y+\varepsilon-t)^{\alpha +1}}t^{n-1} dt  
$$
\begin{equation}\label{11}
=\alpha\int\limits_{0}^{r-\varepsilon}f(T)t^{n-1} dt\int\limits_{t }^{r-\varepsilon  }\frac{( r-y-\varepsilon)^{\alpha-1}   }{(  y+\varepsilon-t)^{\alpha +1}}dy  
$$
$$
=\alpha\int\limits_{0}^{r-\varepsilon}f(T)t^{n-1} dt\int\limits_{t +\varepsilon}^{r  } ( r-y )^{\alpha-1}   (  y -t)^{-\alpha -1} dy .
\end{equation}
Applying formula   (13.18) \cite[p.184]{firstab_lit:samko1987}, we get
\begin{equation}\label{12}
 \int\limits_{t +\varepsilon}^{r  } ( r-y )^{\alpha-1}   (  y -t)^{-\alpha -1} dy=
  \frac{1}{\alpha \varepsilon^{\alpha}}\cdot\frac{(r-t-\varepsilon)^{\alpha}}{ r-t }.
\end{equation}
 Combining    relations \eqref{10},\eqref{11},\eqref{12}, using the change of the variable     $t=r-\varepsilon\tau, $ we get
$$
(\mathfrak{I}^{\alpha}_{0 +}\varphi^{+}_{ \varepsilon }f)(Q)\frac{\pi r^{n-1}}{\sin\alpha\pi}  
$$
$$
=\frac{1}{\varepsilon^{\alpha}}\left\{ \int\limits_{0}^{r  }f (P+y\mathbf{e})(r-y)^{\alpha-1}y^{n-1}   dy- \int\limits_{0}^{r-\varepsilon  } \frac{f(T)(r-t-\varepsilon)^{\alpha}}{ r-t } t^{n-1}dt \right\} 
$$
\begin{equation}\label{13}
=\frac{1}{ \varepsilon^{\alpha}} \int\limits_{0 }^{r  } \frac{f(T)\left[(r -t)^{\alpha}-(r-t-\varepsilon)_{+}^{\alpha}\right]}{ r-t }t^{n-1}dt 
$$
$$
=\int\limits_{0 }^{r/\varepsilon   }\frac{\tau^{\alpha}-(\tau-1)_{+}^{\alpha}}{ \tau } f(P+[r-\varepsilon \tau ]\mathbf{e})(r-\varepsilon \tau)^{n-1} d\tau,\;\tau_{+}=\left\{\begin{array}{cc}\tau,\;\tau\geq0;\\[0,25cm] 0,\;\tau<0\,.\end{array}\right. .
 \end{equation}
Consider the auxiliary function $\mathcal{K}$ defined in the paper   \cite[p.105]{firstab_lit:samko1987}
 \begin{equation}\label{14}
 \mathcal{K}(t)= \frac{\sin\alpha\pi}{\pi }\cdot\frac{ t_{+}^{\alpha}-(t-1)_{+}^{\alpha}}{ t }\in L_{p}(\mathbb{R}^{1});\; \int\limits_{0 }^{\infty  }\mathcal{K}(t)dt=1;\;\mathcal{K}(t)>0.
\end{equation}
  Combining  \eqref{13},\eqref{14} and taking into account that    $f$ has the zero extension outside of  $\bar{\Omega},$ we obtain
 \begin{equation}\label{15}
 (\mathfrak{I}^{\alpha}_{0+}\varphi^{+}_{ \varepsilon }f)(Q)-f(Q) =  \int\limits_{0 }^{\infty  }\mathcal{K}(t) \left\{f(P+[r-\varepsilon t]\mathbf{e})(1-\varepsilon t/r)_{+}^{n-1}-f(P+ r \mathbf{e})  \right\}dt.
\end{equation}
Consider the case  $(0\leq  r <\varepsilon).$ Taking into account \eqref{7}, we get
 \begin{equation}\label{16}
(\mathfrak{I}^{\alpha}_{0+}\varphi^{+}_{ \varepsilon }f)(Q)-f (Q) =
\frac{\sin\alpha\pi}{\pi\varepsilon^{\alpha}} \int\limits_{0}^{r }\frac{f (T)}{(r-t)^{1-\alpha} } \left(\frac{t}{r} \right)^{n-1} dt-f (Q) 
$$
$$
=\frac{\sin\alpha\pi}{\pi\varepsilon^{\alpha}} \int\limits_{0}^{r }\frac{f (P+[r-t]\mathbf{e})}{t^{1-\alpha} }\left(\frac{r-t}{r} \right)^{n-1} dt-f (Q).
\end{equation}
Consider the domains
 \begin{equation}\label{17}
 \Omega_{\varepsilon} :=\{Q\in\Omega,\,d(\mathbf{e})\geq\varepsilon \},\;\tilde{\Omega}_{  \varepsilon }=\Omega\setminus \Omega_{\varepsilon}.
 \end{equation}
In accordance with    this definition  we can  divide the surface $\omega$ into two  parts $ \omega_{\varepsilon}$ and  $\tilde{\omega}_{ \varepsilon },$ where $ \omega_{\varepsilon}$ is  the subset of  $\omega$ such that   $d(\mathbf{e})\geq\varepsilon$ and $\tilde{\omega}_{ \varepsilon }$ is  the subset of  $\omega$ such that  $d(\mathbf{e}) <\varepsilon.$
Using  \eqref{15},\eqref{16}, we get
\begin{equation}\label{18}
  \|(\mathfrak{I}^{\alpha}_{0+}\varphi^{+}_{ \varepsilon }f) -f\|^{p}_{L_{p}(\Omega)}  
$$
$$
=   \int\limits_{\omega_{\varepsilon}}d\chi\int\limits_{\varepsilon}^{ d}
  \left|\int\limits_{0 }^{\infty  }\mathcal{K}(t)[f(Q- \varepsilon t \mathbf{e})(1-\varepsilon t/r)_{+}^{n-1}-f(Q)]dt\right|^{p}r^{n-1}dr    
$$
$$
+ \int\limits_{\omega_{\varepsilon}}d\chi\int\limits_{0}^{\varepsilon }
\left|  \frac{\sin\alpha\pi}{\pi\varepsilon^{\alpha}}\int\limits_{0}^{r }
\frac{f (P+[r-t]\mathbf{e})}{t^{1-\alpha} }\left(\frac{r-t}{r} \right)^{n-1} dt-f (Q)\right|^{p}r^{n-1}dr 
$$
$$
+  \int\limits_{\tilde{\omega}_{\varepsilon} }d\chi\int\limits_{0}^{d }
\left| \frac{\sin\alpha\pi}{\pi\varepsilon^{\alpha}}\int\limits_{0}^{r }\frac{f (P+[r-t]\mathbf{e})}{t^{1-\alpha} }\left(\frac{r-t}{r} \right)^{n-1} dt-f (Q) \right|^{p}r^{n-1}dr =I_{1}+I_{2}+I_{3}.
 \end{equation}
Consider $ I_{1},$ using the generalized Minkovski  inequality,  we get
$$
  I^{\frac{1 }{p}}_{1} \leq  \int\limits_{0 }^{\infty  }\mathcal{K}(t)
  \left(\int\limits_{\omega_{\varepsilon} }d\chi\int\limits_{\varepsilon}^{ d}
  |f(Q- \varepsilon t \mathbf{e})(1-\varepsilon t/r)_{+}^{n-1}-f(Q)|^{p}r^{n-1}dr \right)^{\frac{1 }{p}} dt.
$$
We use the following  notation
$$
 h(\varepsilon,t):= \mathcal{K}(t)\left(\int\limits_{\omega_{\varepsilon} }d\chi\int\limits_{\varepsilon}^{ d}
  |f(Q- \varepsilon t \mathbf{e})(1-\varepsilon t/r)_{+}^{n-1}-f(Q)|^{p}r^{n-1}dr \right)^{\frac{1 }{p}} dt.
$$
It can easily be checked that
\begin{equation}\label{19}
  |h(\varepsilon,t)|\leq 2\mathcal{K}(t) \| f\|_{L_{p}(\Omega)},\;\forall\varepsilon>0;
\end{equation}
  \begin{equation*}
  |h(\varepsilon,t)|\leq  \left(\int\limits_{\omega_{\varepsilon} }d\chi\int\limits_{\varepsilon}^{ d}
 \left |(1-\varepsilon t/r)_{+}^{n-1}[f(Q- \varepsilon t \mathbf{e})-f(Q)]\right|^{p}r^{n-1}dr \right)^{\frac{1 }{p}} dt 
 $$
 $$+\left(\int\limits_{\omega_{\varepsilon} }d\chi\int\limits_{0}^{ d}
  \left|f(Q)  [1-(1-\varepsilon t/r)_{+}^{n-1}]\right|^{p}r^{n-1}dr \right)^{\frac{1 }{p}} dt=I_{11}+I_{12}.
\end{equation*}
By virtue of the average continuity property  in  $L_{p}(\Omega),$  we have $\forall t>0:\, I_{11}\rightarrow 0,\;\varepsilon\downarrow 0.$
Consider $I_{12}$ and  let us define the function
$$
h_{1}(\varepsilon,t,r):=\left|f(Q)  [1-(1-\varepsilon t/r)_{+}^{n-1}]\right|.
$$ Obviously, the following relations hold  almost everywhere  in $\Omega$
$$
\forall t>0,\,h_{1}(\varepsilon,t,r) \leq |f(Q)|,\;h_{1}(\varepsilon,t,r)\rightarrow 0,\;\varepsilon \downarrow0.
$$
Applying the Lebesgue  dominated convergence theorem, we get $I_{12}\rightarrow 0,\;\varepsilon\downarrow 0.$
It implies that
\begin{equation}\label{20}
\forall t>0,\,\lim\limits_{\varepsilon\rightarrow 0} h(\varepsilon,t)=0.
\end{equation}
Taking into account  \eqref{19}, \eqref{20} and    applying  the Lebesgue  dominated convergence theorem again, we obtain
$$
 I_{1}\rightarrow 0,\;\;\varepsilon\downarrow0  .
$$
Consider $I_{2},$ using the Mincovski  inequality, we  get
$$
I^{\frac{1 }{p}}_{2} \leq \frac{\sin\alpha\pi}{\pi\varepsilon^{\alpha}}\left( \int\limits_{\omega_{\varepsilon} }d\chi\int\limits_{0}^{\varepsilon }
\left| \int\limits_{0}^{r }\frac{f (Q-t\mathbf{e})}{t^{1-\alpha} }\left(\frac{r-t}{r} \right)^{n-1} dt\right|^{p}r^{n-1}dr\right)^{\frac{1 }{p}}
$$
$$
+\left(\int\limits_{\omega_{\varepsilon} }d\chi\int\limits_{0}^{\varepsilon}
\left| f (Q) \right|^{p}r^{n-1}dr \right)^{\frac{1 }{p}}=I_{21} +I_{22}.
$$
Applying the generalized  Mincovski  inequality, we obtain
$$
I_{21}\frac{\pi}{\sin\alpha\pi}=\frac{1}{\varepsilon^{\alpha}}\left(\int\limits_{\omega_{\varepsilon} }d\chi
\int\limits_{0}^{\varepsilon }
\left|  \int\limits_{0}^{r }\frac{f (Q-t\mathbf{e})}{t^{1-\alpha} }\left(\frac{r-t}{r} \right)^{n-1} \!\!\! dt\right|^{p}r^{n-1}\! dr \right)^{\frac{1 }{p}} 
$$
$$
\leq\frac{1}{\varepsilon^{\alpha}}\left\{\int\limits_{\omega_{\varepsilon} }\!\!\left[\int\limits_{0}^{\varepsilon }\!\!t ^{\alpha-1 }\!\!
\left( \int\limits_{t}^{\varepsilon }\!\!|f (Q -t \mathbf{e})|^{p}\!\left(\frac{r-t}{r} \right)^{\!\!\!(p-1)(n-1)}\!\!\!(r-t)^{n-1} \!dr  \right)^{\frac{1 }{p}}\!\!dt\right]^{p}\!\!d\chi \right\}^{\frac{1 }{p}} 
 $$
 $$
\leq\frac{1}{\varepsilon^{\alpha}}\left\{\int\limits_{\omega_{\varepsilon} }\left[\int\limits_{0}^{\varepsilon } t ^{\alpha-1 }
\left( \int\limits_{t}^{\varepsilon } \left|f (P+[r-t]\mathbf{e})\right|^{p}  (r-t)^{n-1}dr  \right)^{\frac{1 }{p}}\!\!dt\right]^{p}\!\!d\chi \right\}^{\frac{1 }{p}} 
 $$
  $$
\leq\frac{1}{\varepsilon^{\alpha}}\left\{\int\limits_{\omega_{\varepsilon} }\left[\int\limits_{0}^{\varepsilon } t ^{\alpha-1 }
\left( \int\limits_{0}^{\varepsilon } |f (P +r \mathbf{e})|^{p}  r^{n-1}dr  \right)^{\frac{1 }{p}}\!\!dt\right]^{p}\!\!d\chi \right\}^{\frac{1 }{p}}\!\!=\frac{1}{\alpha}\| f\| _{L_{p}( \Delta_{\varepsilon})},\;
$$
$$
\Delta_{\varepsilon} :=\{Q\in\Omega_{\varepsilon},\,r<\varepsilon \} .
 $$
Note that   ${\rm mess}\, \Delta_{\varepsilon}\rightarrow 0,\,\varepsilon\downarrow0,$
hence  $I_{21},I_{22}\rightarrow 0,\,\varepsilon\downarrow0 .$ It follows  that $I_{2 }\rightarrow 0,\,\varepsilon\downarrow0.$ In the same way, we obtain   $I_{3 }\rightarrow 0,\,\varepsilon\downarrow 0.$ Since we proved that $I_{1},I_{2},I_{3}\rightarrow 0,\,\varepsilon\downarrow0,$ then relation \eqref{10.01} holds.
 This completes the proof corresponding to the left-side case.
 The proof corresponding to the right-side case is absolutely analogous.
\end{proof}
\begin{theorem}\label{T3}
Suppose  $f=\mathfrak{I}^{\alpha}_{0+} \psi$ or $f= \mathfrak{I} ^{\alpha}_{d -} \psi ,\;\psi\in L_{p}(\Omega),\;1\leq p<\infty;$ then
$
\, \mathfrak{D }^{\alpha}_{0+}f =\psi
$
or
$
 \mathfrak{D} ^{\alpha}_{d-}f =\psi
$
respectively.
 \end{theorem}
\begin{proof}
Consider
$$
r^{n-1}f(Q)-(r-\tau)^{n-1}f(Q-\tau\mathbf{e}) 
$$
$$
=\frac{1}{\Gamma(\alpha)} \int\limits_{0}^{ r}\frac{\psi (Q-t\mathbf{e})}{t^{1-\alpha}}\left( r-t \right)^{n-1}dt-\frac{1}{\Gamma(\alpha)} \int\limits_{\tau}^{ r}\frac{\psi (Q-t\mathbf{e})}{(t-\tau)^{1-\alpha}}\left( r-t \right)^{n-1}dt 
$$
$$
 =\tau^{\alpha-1}  \int\limits_{0}^{ r} \psi (Q-t\mathbf{e}) k\left(\frac{t}{\tau}\right)(r-t)^{n-1}dt,\;k(t)= \frac{1}{\Gamma(\alpha)}\left\{\begin{array}{cc}t^{\alpha-1},\;0<t<1;\\[0,25cm] t^{\alpha-1}-(t-1)^{\alpha-1},\;t>1.\end{array}\right.
$$
Hence in the case    $(\varepsilon\leq r\leq d),$ we have
$$
(\psi^{+}_{  \varepsilon }f)(Q)=\int\limits_{ \varepsilon }^{ r }\frac{r^{n-1} f(Q)-(r-\tau)^{n-1}f(Q-\tau\mathbf{e})}{ r^{n-1}\tau ^{\alpha +1}} d\tau 
$$
$$
=\int\limits_{ \varepsilon }^{ r } \tau^{-2} d\tau\int\limits_{0}^{ r} \psi (Q-t\mathbf{e}) k\left(\frac{t}{\tau}\right)\left( 1-t/r \right)^{n-1}dt  
$$
$$
=\int\limits_{ 0 }^{ r }\psi (Q-t\mathbf{e}) \left( 1-t/r \right)^{n-1}  dt\int\limits_{\varepsilon}^{ r}  k\left(\frac{t}{\tau}\right) \tau^{-2}d \tau   
$$
$$
=\int\limits_{ 0 }^{ r }\psi (Q-t\mathbf{e}) \left( 1-t/r \right)^{n-1}t^{-1} dt\int\limits_{ t/ r  }^{t/\varepsilon}  k (s )ds.
$$
Applying   formula    (6.12) \cite[p.106]{firstab_lit:samko1987}, we get
$$
(\psi^{+}_{  \varepsilon }f)(Q)\cdot\frac{\alpha}{\Gamma(1-\alpha)}=\int\limits_{ 0 }^{ r }\psi (Q-t\mathbf{e})
 \left( 1-t/r \right)^{n-1}\left[ \frac{1}{\varepsilon}\mathcal{K}\left(\frac{t}{\varepsilon}\right) - \frac{1}{ r}\mathcal{K}\left(\frac{t}{ r}\right) \right]dt.
$$
Since in accordance with \eqref{14},  we have
 $$
\mathcal{K}\left(\frac{t}{ r}\right)=\left[\Gamma(1-\alpha)\Gamma(\alpha) \right]^{-1} \left(\frac{t}{ r}\right)^{\alpha-1}\!\!\!,
$$
then
$$
(\psi^{+}_{  \varepsilon }f)(Q)\cdot\frac{\alpha}{\Gamma(1-\alpha)}=\int\limits_{ 0 }^{ r /\varepsilon }\mathcal{K} ( t   ) \psi (Q-\varepsilon t\mathbf{e})\left( 1-\varepsilon t/r \right)^{n-1}
  dt-\frac{f(Q)}{\Gamma(1-\alpha) r ^{ \alpha}}.
$$
 Taking into account \eqref{8},\eqref{14}, and that the function     $\psi(Q)$ has the zero  extension   outside of   $\bar{\Omega},$   we obtain
$$
 ( \mathfrak{D} ^{\alpha}_{0+,\varepsilon}f)(Q)-\psi(Q) =\int\limits_{ 0 }^{\infty}
\mathcal{K} ( t   ) \left[\psi (Q - \varepsilon t \mathbf{e})(1-\varepsilon t/r)_{+}^{n-1}-\psi (Q) \right]dt,\;\varepsilon\leq r\leq d  .
$$
Consider the case    $ ( 0\leq r<\varepsilon).$ In accordance with \eqref{7},  we have
$$
( \mathfrak{D} ^{\alpha}_{0+,\varepsilon}f)(Q)-\psi(Q)=\frac{f(Q)}{\varepsilon^{\alpha}\Gamma(1-\alpha)}-\psi(Q).
$$
Using the generalized Mincovski  inequality, we   get
$$
\|( \mathfrak{D} ^{\alpha}_{0+,\varepsilon}f)(Q)-\psi(Q)\|_{L_{p}(\Omega)}\leq\!\! \int\limits_{ 0 }^{\infty}
\!\! \mathcal{K}(t)   \| \psi(Q - \varepsilon t \mathbf{e})(1-\varepsilon t/r)_{+}^{n-1}-\psi (Q)\|_{L_{p}(\Omega)}dt 
$$
$$
+\frac{1}{ \Gamma(1-\alpha)\varepsilon^{\alpha}}\|f\|_{L_{p}(\Delta'_{\varepsilon})}+\|\psi\|_{L_{p}(\Delta'_{\varepsilon})},\;\Delta'_{\varepsilon}=\Delta_{\varepsilon}\cup \tilde{\Omega}_{\varepsilon},
$$
here we use the denotations that were used in Theorem \ref{T2}. Arguing as above (see Theorem \ref{T2}), we see that
    all three  summands of the right side  of the previous  inequality tend to zero, when  $\varepsilon\downarrow 0.$
\end{proof}

\begin{theorem} \label{T4}  Suppose $\rho\in {\rm Lip}\,\lambda,\;\alpha<\lambda\leq 1,\;f\in H_{0}^{1}(\Omega) ;$ then $\rho f\in  \mathfrak{I} ^{\alpha}_{\,0 +}(L_{2} ) \cap \mathfrak{I} ^{\alpha}_{d -}(L_{2} ).$ \end{theorem}

\begin{proof}
We   provide a proof only for the left-side   case,   the proof corresponding to the right-side case is absolutely analogous.     First,  assume that  $f\in C_{0}^{\infty}(\Omega) .$
Using  the denotations that were used in Theorem \ref{T2},
  we have
\begin{equation}\label{21}
\|\psi^{+}_{ \varepsilon_1}f  - \psi^{+}_{ \varepsilon_2}f \|_{L_{2}(\Omega )}\leq\!\!\|\psi^{+}_{ \varepsilon_1}f  - \psi^{+}_{ \varepsilon_2}f \|_{L_{2}( \Omega_{\varepsilon_{1}})}+\|\psi^{+}_{ \varepsilon_1}f  - \psi^{+}_{ \varepsilon_2}f \|_{L_{2}(\tilde{\Omega}_{ \varepsilon_{1}})},
\end{equation}
where $\varepsilon_{1}>\varepsilon_{2}>0.$
   We have the following reasoning
$$\|\psi^{+}_{ \varepsilon_1}f  - \psi^{+}_{ \varepsilon_2}f \|_{L_{2}(\Omega_{\varepsilon_{1}})}\leq
 \left(\int\limits_{\omega_{\varepsilon_{1}}}d\chi\int\limits_{\varepsilon_{1}}^{d }\left|\int\limits_{r-\varepsilon_{1}}^{r-\varepsilon_{2}}\frac{(\rho f)(Q)r^{n-1}-
 (\rho f)(T)t^{n-1}}{r^{n-1}(r-t)^{\alpha +1}} dt \right|^{2}r^{n-1} dr    \right)^{\frac{1}{2}} 
$$
$$
  +\!\!\left(\int\limits_{\omega_{\varepsilon_{1}}}d\chi\int\limits_{ \varepsilon_{2}}^{ \varepsilon_{1}}\!\!\left|\int\limits_{0}^{r- \varepsilon_{1}}\frac{(\rho f)(Q)r^{n-1} }{r^{n-1}(r-t)^{\alpha +1}} dt-\!\!
 \int\limits_{ 0}^{r-\varepsilon_{2}}\frac{(\rho f)(Q)r^{n-1}-(\rho f)(T)t^{n-1}}{r^{n-1}(r-t)^{\alpha +1}} dt\right|^{2}r^{n-1} dr    \right)^{\frac{1}{2}} 
 $$
 $$
 +\left(\int\limits_{\omega_{\varepsilon_{1}}}d\chi\int\limits_{0}^{\varepsilon_{2} }\left|\int\limits_{ 0}^{r-\varepsilon_{1}}
 \frac{(\rho f)(Q)r^{n-1} }{r^{n-1}(r-t)^{\alpha +1}} dt-
 \int\limits_{0}^{r- \varepsilon_{2}}\frac{(\rho f)(Q)r^{n-1} }{r^{n-1}(r-t)^{\alpha +1}} dt\right|^{2}r^{n-1} dr    \right)^{\frac{1}{2}} 
$$
$$
 =I_{1}+I_{2}+I_{3}.
$$
  Since  $f\in C_{0}^{\infty}(\Omega),$ then for sufficiently  small   $\varepsilon_{1}>0$ we have  $f(Q)=0,\,r<\varepsilon_{1}.$ This implies that   $I_{2}=I_{3}=0$ and that the second summand of the right side of   inequality   \eqref{21} equals zero.
Making the change the variable in $I_{1},$ then using the generalized Minkowski  inequality, we get
$$
 I_{1}  =\left(\int\limits_{\omega_{\varepsilon_{1}}}d\chi\int\limits_{\varepsilon_{1}}^{d}\left|\int\limits_{ \varepsilon_{1}}^{ \varepsilon_{2}}
 \frac{(\rho f)(Q)r^{n-1}-
 (\rho  f )(Q-\mathbf{e} t )(r-t)^{n-1}}{ r^{n-1}t ^{\alpha +1}} dt\right|^{2}r^{n-1} dr    \right)^{\frac{1}{2}} 
$$
$$ \leq\int\limits_{ \varepsilon_{2}}^{ \varepsilon_{1}}t^{-\alpha -1}\!\!\!\left(\int\limits_{\omega_{\varepsilon_{1}}}d \chi\int\limits_{\varepsilon_{1}}^{d}\left| (\rho  f )(Q)-(1-t/r)^{n-1}(\rho  f )(Q-\mathbf{e}t) \right|^{2}  r^{n-1} dr    \right)^{\frac{1}{2}}\!\! dt  
$$
$$
  \leq\int\limits_{ \varepsilon_{2}}^{ \varepsilon_{1}}t^{-\alpha -1}\!\!\!\left(\int\limits_{\omega_{\varepsilon_{1}}}d \chi\int\limits_{\varepsilon_{1}}^{d}\left| (\rho f )(Q)- (\rho f )(Q-\mathbf{e}t) \right|^{2}  r^{n-1} dr    \right)^{\frac{1}{2}} \!\!dt 
$$
$$
+\int\limits_{ \varepsilon_{2}}^{ \varepsilon_{1}}t^{-\alpha -1}\left(\int\limits_{\omega_{\varepsilon_{1}}}d \chi\int\limits_{\varepsilon_{1}}^{d}\left[ 1-( 1-t/r )^{n-1}\right]\left| (\rho f )(Q-\mathbf{e}t) \right|^{2}  r^{n-1} dr    \right)^{\frac{1}{2}}\!\!dt  
$$
$$
\leq C_{1}\!\!\int\limits_{ \varepsilon_{2}}^{ \varepsilon_{1}}t^{\lambda-\alpha-1 }d t+\int\limits_{ \varepsilon_{2}}^{ \varepsilon_{1}}t^{-\alpha }\left(\int\limits_{\omega_{\varepsilon_{1}}}d \chi\int\limits_{\varepsilon_{1}}^{d}
\left|\frac{1}{r} \sum\limits_{i=0}^{n-2}\left( \frac{t}{r}\right)^{i}(\rho f )(Q-\mathbf{e}t) \right|^{2}  r^{n-1} dr    \right)^{\!\!\frac{1}{2}}\!\!dt.
$$
Using  the  function $f$    property, we see that there  exists a constant $\delta$ such that $  f  (Q-\mathbf{e}t )=0,\;r<\delta.$  In accordance with the above reasoning,      we have
 $$
 I_{1} \leq
   C_{1}\frac{  \varepsilon^{  \lambda-\alpha}_{1}-\varepsilon^{\lambda-\alpha}_{2}   }{\!\!\!\!\lambda-\alpha  }+ \|f\|_{L_{2}(\Omega)}\frac{  \varepsilon^{  1- \alpha}_{1}-\varepsilon^{1- \alpha}_{2}   }{\!\!\!\delta(1-\alpha)  } (n-1) .
$$
Applying Theorem \ref{T1}, we complete  the     proof for the case    $ ( f\in C_{0}^{\infty}(\Omega)).$
Now assume that  $f\in H^1_{0}(\Omega),$ then there   exists the sequence $\{f_{n}\} \subset C_{0}^{\infty}(\Omega),\;   f_{n}\stackrel{ H^1_{0}}{\longrightarrow} f.$ It is easy to prove that $   \rho f_{n}\stackrel{ L_{2}}{\longrightarrow} \rho f.$
In accordance with the proven  above fact, we have $\rho f_{n}= \mathfrak{I} ^{\alpha}_{0+}\varphi_{n},\;\{\varphi_{n}\}\in L_{2}(\Omega),$ therefore
\begin{equation}\label{22}
 \mathfrak{I} ^{\alpha}_{0+}\varphi_{n}\stackrel{L_{2} }{\longrightarrow} \rho f.
\end{equation}
To conclude the proof, it is sufficient to show that   $\varphi_{n}\stackrel{L_{2} }{\longrightarrow}\varphi\in L_{2}(\Omega).$  Note that by virtue of  Theorem  \ref{T2} we have  $\mathfrak{ D} ^{\alpha}_{0+}\rho f_{n}=\varphi_{n}.$
Let  $ c_{n,m}:=f_{n+m}-f_{n},$  we have
$$
\|\varphi_{n+m}-\varphi_{n}\|_{L_{2}(\Omega)}\leq\frac{\alpha}{\Gamma(1-\alpha)}\left(\int\limits_{\Omega} \left|\int\limits_{0 }^{r}\frac{(\rho c_{n,m})(Q)r^{n-1}-
(\rho c_{n,m})(T)t^{n-1}}{r^{n-1}( t-r)^{\alpha +1}} dt\right|^{2} dQ    \right)^{\frac{1}{2}} 
$$
$$
+\frac{1}{\Gamma(1-\alpha)}\left(\int\limits_{\Omega}\left| \frac{(\rho c_{n,m})(Q) }{ r ^{\alpha }} \right|^{2}dQ    \right)^{\frac{1}{2}}=I_3 +I_4.
$$
Consider $I_{3}.$ It can be shown in the usual way that
$$
\frac{\Gamma(1-\alpha)}{\alpha} I_3  \leq
  \left\{\int\limits_{\Omega}\left|\int\limits_{0 }^{ r}\frac{  (\rho c_{n,m}) (Q)-  (\rho c_{n,m}) (Q-\mathbf{e}t) }{ t^{\alpha +1}} dt\right|^{2} dQ    \right\}^{\frac{1}{2}} 
$$
$$
  + \left\{\int\limits_{\Omega} \left|\int\limits_{0 }^{ r}\frac{(\rho c_{n,m})(Q-\mathbf{e}t)[1-(1- t/r)^{n-1} ]  }{ t^{1+\alpha   }} dt\right|^{2} dQ    \right\}^{\frac{1}{2}}=
    I_{01}+I_{02};
 $$
 $$
I_{01}\leq \sup\limits_{Q\in \Omega}|\rho(Q)|\left\{\int\limits_{\Omega}\left(\int\limits_{0 }^{ r}\frac{  |  c_{n,m} (Q)-    c_{n,m} (Q-\mathbf{e}t) |}{ t^{\alpha +1}} dt\right)^{2} dQ    \right\}^{\frac{1}{2}} 
$$
$$
+\left\{\int\limits_{\Omega}\left|\int\limits_{0 }^{ r}\frac{ c_{n,m} (Q-\mathbf{e}t)  [  \rho (Q)-    \rho (Q-\mathbf{e}t)] }{ t^{\alpha +1}} dt\right|^{2} dQ    \right\}^{\frac{1}{2}}=
 I_{11}+I_{21} .
$$
Applying  the generalized Minkowski  inequality, then representing the function under the inner integral by the directional derivative,   we get
$$
I_{11}
 \leq C_{1}\int\limits_{ 0}^{ \mathfrak{d}}t^{-\alpha -1}\left(\int\limits_{\Omega} \left| c_{n,m}(Q)-c_{n,m}(Q-\mathbf{e}t) \right|^{2}  dQ   \right)^{\frac{1}{2}} dt 
$$
$$
=C_{1}\int\limits_{ 0}^{ \mathfrak{d}}t^{-\alpha -1}\left(\int\limits_{\Omega} \left| \int\limits_{0}^{t} c'_{n,m} (Q-\mathbf{e}\tau) d\tau\right|^{2}   dQ\right)^{\frac{1}{2}} dt.
$$
Using the Cauchy-Schwarz inequality, the Fubini  theorem, we have
$$
I_{11} \leq C_{1}\int\limits_{ 0}^{ \mathfrak{d}}t^{-\alpha -1}\left(\int\limits_{\Omega}dQ \int\limits_{0}^{t} \left|c'_{n,m} (Q-\mathbf{e}\tau)\right|^{2}d\tau \int\limits_{0}^{t}d\tau    \right)^{\frac{1}{2}} dt 
$$

$$
  =C_{1}\int\limits_{ 0}^{ \mathfrak{d}}t^{-\alpha -1/2 }\left( \int\limits_{0}^{t}d\tau \int\limits_{\Omega}\left|c'_{n,m} (Q-\mathbf{e}\tau)\right|^{2} dQ   \right)^{\frac{1}{2}} dt\leq C_{1}
 \frac{ \mathfrak{d}^{1-\alpha}  }{1-\alpha  } \,  \|c'_{n,m} \|_{L_{2}(\Omega)}.
$$
 Arguing as above, using the Holder  property of the function $\rho,$ we see that
$$
I_{21}\leq M \int\limits_{ 0}^{ \mathfrak{d}}t^{\lambda-\alpha -1}\left(\int\limits_{\Omega} \left|  c_{n,m}(Q-\mathbf{e}t) \right|^{2}  dQ    \right)^{\frac{1}{2}} dt\leq
M\frac{ \mathfrak{d}^{\lambda-\alpha}  }{\lambda-\alpha } \,  \|c _{n,m} \|_{L_{2}(\Omega)}.
$$
It can be shown in the usual way that
$$
I_{02}\leq C_{1}\left\{\int\limits_{\Omega} \left|\int\limits_{0 }^{ r}  | c_{n,m}(Q-\mathbf{e}t)| \sum\limits_{i=0}^{n-2}\left(\frac{t}{r}\right)^{i}  r^{-1} t^{ -\alpha   }  dt\right|^{2} dQ    \right\}^{\frac{1}{2}} 
$$
 $$
\leq C_{2}\left\{\int\limits_{\Omega} \left(\int\limits_{0 }^{ r} t^{ -\alpha   }dt \int\limits_{t}^{r}\left| c'_{n,m}(Q-\mathbf{e}\tau)\right| d\tau      \right)^{2}r^{-2}  dQ    \right\}^{\frac{1}{2}} 
$$
$$
= C_{2}\left\{\int\limits_{\Omega} \left(\int\limits_{0 }^{ r}\left| c'_{n,m}(Q-\mathbf{e}\tau)\right| d\tau  \int\limits_{0}^{\tau}t^{ -\alpha   }dt      \right)^{2}r^{-2}  dQ    \right\}^{\frac{1}{2}} 
$$
 $$
 \leq\frac{C_{2}}{1-\alpha}\left\{\int\limits_{\Omega} \left(\int\limits_{0 }^{ r}  \left| c'_{n,m}(Q-\mathbf{e}\tau)\right| \tau^{ -\alpha} d\tau      \right)^{2}   dQ    \right\}^{\frac{1}{2}}.
$$
Applying  the generalized Minkowski  inequality, we have
$$
I_{02}\leq C_{3}\int\limits_{0 }^{ \mathfrak{d}} \tau^{ -\alpha} d\tau  \left(\int\limits_{\Omega} \left| c'_{n,m}(Q-\mathbf{e}\tau)\right|^{2} dQ     \right)^{\frac{1}{2}}   \leq
C_{3}\frac{\mathfrak{d}^{1-\alpha}}{1-\alpha} \|c'_{n,m}\|_{L_{2}(\Omega)}.
$$
Consider   $I_{2},$  we have  
$$
I_2  \leq\frac{C_{ 1}}{\Gamma(1-\alpha)}\left(\int\limits_{\Omega} \left| c_{n,m}(Q)\right|^{2}   r ^{-2\alpha  } dQ\right)^{ \frac{1}{2}}  
$$
$$
=\frac{C_{ 1}}{\Gamma(1-\alpha)} \left(\int\limits_{\Omega}  r ^{-2\alpha  } \left|\int\limits_{0}^{r} c'_{n,m}(Q-\mathbf{e}t)dt\right|^{2}dQ     \right)^{\frac{1}{2}} 
 $$
 $$
 \leq\frac{C_{ 1}}{\Gamma(1-\alpha)} \left(\int\limits_{\Omega}    \left|\int\limits_{0}^{r} c'_{n,m}(Q-\mathbf{e}t)t^{-\alpha}dt\right|^{2}dQ     \right)^{\frac{1}{2}}.
$$
Using the generalized Minkowski inequality, then applying the  trivial estimates,  we get
$$
 I_2\leq  C_{ 4} \left\{\int\limits_{\omega}\left[\int\limits_{0}^{d} t^{-\alpha} dt  \left(\int\limits_{t}^{d}|c'_{n,m}(Q-\mathbf{e}t)|^{2}   r^{ n-1 }dr\right)^{\frac{1}{2}} \right]^{2}d\chi  \right\}^{\frac{1}{2}}  
$$
$$
  \leq  C_{ 4} \left\{\int\limits_{\omega}\left[\int\limits_{0}^{\mathfrak{d}} t^{-\alpha} dt  \left(\int\limits_{0}^{d}|c'_{n,m}(Q-\mathbf{e}t)|^{2}   r^{ n-1 }dr\right)^{\frac{1}{2}} \right]^{2}d\chi  \right\}^{\frac{1}{2}}  
$$
$$
  =  C_{ 4}\int\limits_{0}^{\mathfrak{d}} t^{-\alpha}   dt  \left(\int\limits_{\omega}d\chi\int\limits_{0}^{d}|c'_{n,m}(Q-\mathbf{e}t)|^{2}   r^{ n-1 }dr\right)^{\frac{1}{2}}       \leq  C_{ 4}\frac{\mathfrak{d}^{1-\alpha}}{1-\alpha}\|c'_{n,m}\|_{L_{2}(\Omega)}.
$$
  Taking into account that the sequences    $\{f_{n}\},\{f'_{n}\}$ are fundamental, we obtain       $I_{1},I_{2}\rightarrow 0.$ Hence the  sequence  $\{\varphi_{n}\} $ is fundamental and  $\varphi_{n}\stackrel{L_{2} }{\longrightarrow}\varphi\in L_{2}(\Omega).$  Note that by virtue  of  Theorem \ref{T1} the  directional fractional integral   operator  is bounded on the space $L_{2}(\Omega).$ Hence
$$
 \mathfrak{I }^{\alpha}_{0 +}\varphi_{n}\stackrel{L_{2} }{\longrightarrow}  \mathfrak{I} ^{\alpha}_{0+}\varphi.
$$
Combining this fact with   \eqref{22}, we have   $\rho f=  \mathfrak{I} ^{\alpha}_{0+}\varphi.$

\end{proof}
 \begin{lemma}\label{L1}
 The operator $\mathfrak{D}^{ \alpha }$ is a restriction  of the  operator $\mathfrak{D}^{ \alpha }_{0+}.$
\end{lemma}
\begin{proof}
We need to show that the  next equality holds
\begin{equation}\label{23}
 (\mathfrak{D}^{ \alpha }f)(Q)= \left(\mathfrak{D}^{ \alpha }_{0+} f\right)(Q), \,f\in \stackrel{\circ}{W_p ^l} (\Omega).
  \end{equation}
It can be shown in the usual way that
$$
 \mathfrak{D}^{ \alpha }v=\frac{\alpha}{\Gamma(1-\alpha)}\int\limits_{0}^{r} \frac{v(Q)-v(T)}{(r- t)^{\alpha+1}}  \left(\frac{t}{r}\right)  ^{n-1} dt+
  C_{n}^{(\alpha)}    v(Q)   r ^{ -\alpha} 
$$
$$
= \frac{\alpha}{\Gamma(1-\alpha)}\int\limits_{0}^{r} \frac{r^{n-1}v(Q)-t^{n-1}v(T)}{r^{n-1}(r- t)^{\alpha+1}}dt-\frac{\alpha\,v(Q)}{\Gamma(1-\alpha)}\int\limits_{0}^{r}
\frac{ r^{n-1} -t^{n-1}  }{r^{n-1}(r- t)^{\alpha+1}}   dt 
$$
$$
+C_{n}^{(\alpha)}   v(Q)   r ^{  -\alpha}=
  ( \mathfrak{D} ^{\alpha}_{0+}v)(Q)-
\frac{\alpha \, v(Q) }{\Gamma(1-\alpha)} \sum\limits_{i=0}^{n-2}  r ^{ -1-i}\int\limits_{0}^{r} \frac{t^{i}}{(r- t)^{\alpha }}   dt 
$$
\begin{equation}\label{24}
+C_{n}^{(\alpha)}  v(Q)   r ^{  -\alpha}-\frac{v(Q)   r ^{  -\alpha}}{\Gamma(1-\alpha)}   =  ( \mathfrak{D} ^{\alpha}_{0+}v)(Q)- I_1 +I_2 -I_3.
\end{equation}
  Using the formula of the fractional integral of a power  function (2.44) \cite[p.47]{firstab_lit:samko1987}, we have
$$
I_1 =
\frac{\alpha\, v(Q)\,r ^{-1 } }{\Gamma(1-\alpha)}\int\limits_{0}^{r} \frac{dt}{(r- t)^{\alpha }}      +\frac{\alpha\, v(Q) }{\Gamma(1-\alpha)} \sum\limits_{i=1}^{n-2}
r ^{-1-i}\int\limits_{0}^{r} \frac{t^{i}}{(r- t)^{\alpha }}   dt 
$$
$$
  =
 v(Q)\frac{\alpha }{ \Gamma(2-\alpha)}r ^{ -\alpha }    +  v(Q)  \alpha \sum\limits_{i=1}^{n-2}
r ^{-1-i} (I^{1-\alpha}_{0+}t^{i})(r)     
$$
$$
=v(Q)\frac{\alpha }{ \Gamma(2-\alpha)}r ^{ -\alpha }    +  v(Q)  \alpha \sum\limits_{i=1}^{n-2}
r ^{ -\alpha}\frac{i!}{\Gamma(2-\alpha+i)}.
$$
Hence
$$
  I_1 +I_3
=\frac{v(Q)r ^{ -\alpha }}{ \Gamma(2-\alpha)}    +   v(Q)r ^{ -\alpha }  \alpha \sum\limits_{i=1}^{n-2}
 \frac{i!}{\Gamma(2-\alpha+i)}=
 \frac{2 v(Q)r ^{ -\alpha }}{ \Gamma(3-\alpha)}      
 $$
 $$
 +   v(Q)r ^{ -\alpha } \alpha \sum\limits_{i=2}^{n-2}
 \frac{i!}{\Gamma(2-\alpha+i)}=
 \frac{3!v(Q)r ^{ -\alpha }}{ \Gamma(4-\alpha)}     +    v(Q)r ^{ -\alpha } \alpha \sum\limits_{i=3}^{n-2}
 \frac{i!}{\Gamma(2-\alpha+i)} 
$$
\begin{equation}\label{25}
 =\frac{(n-2)!v(Q)r ^{ -\alpha }}{ \Gamma(n-1-\alpha)}     +    v(Q)r ^{ -\alpha } \alpha
 \frac{(n-2)!}{\Gamma(n-\alpha )}=C_{n}^{(\alpha)}v(Q)r ^{ -\alpha }.
\end{equation}
Therefore $I_{2}-I_{1}-I_{3}=0$
and obtain   equality \eqref{23}. Let us prove that the considered operators do  not coincide with each other. For this purpose consider the function
  $f= \mathfrak{I}^{\alpha}_{0+}  \varphi,\;\varphi \in L_{p}(\Omega) ,$  then in accordance with  Theorem \ref{T2}, we have $ \mathfrak{D} _{0+} ^{ \alpha }  \mathfrak{I}^{\alpha}_{0+}   \varphi  =\varphi.$ Hence  $\mathfrak{I}^{\alpha}_{0+} \left( L_{p}\right)\subset\mathrm{D}\left(\mathfrak{D}^{ \alpha }_{0+}\right).$ Now  it is sufficient to notice  that
$$
 \exists f\in \mathfrak{I}^{\alpha}_{0+} \left( L_{p}\right) ,\;f(\Lambda)\neq 0,
$$
where $\Lambda \subset\partial\Omega,\;{\rm mess}\, \Lambda\neq 0.$
On the other hand, we know that
$$
f(\partial \Omega)= 0\;a.e.,\;\forall f\in \mathrm{D}\left(\mathfrak{D}^{ \alpha }  \right) .
$$

\end{proof}
\begin{lemma}\label{L2} The following identity  holds
\begin{equation*}
  (\mathfrak{D} _{0+} ^{ \alpha })^{*}  = \mathfrak{D }^{\alpha}_{d-},
\end{equation*}
where    limits \eqref{8.1} are understood as  the limits with respect to the $L_{2}(\Omega)$ norm.
 \end{lemma}
\begin{proof}
Let us  show that the next relation is true
\begin{equation}\label{26}
 (\mathfrak{D}^{ \alpha }_{0+} f ,  g  )_{L_{2}(\Omega)}
=(f , \mathfrak{D }^{\alpha}_{d-} g )_{L_{2}(\Omega)},
\end{equation}
$$
 f\in \mathfrak{I}^{\alpha}_{0+} \left( L_{2}\right),\,g\in  \mathfrak{I } ^{\alpha}_{d-}\left(  L_{2}\right) .
$$
Note that by virtue of Theorem  \ref{T3}, we have    $\,\mathfrak{D}^{ \alpha }_{0+} \mathfrak{I}^{\alpha}_{0+} \varphi   =\varphi ,\,
\mathfrak{D }^{\alpha}_{d-}  \mathfrak{I } ^{\alpha}_{d-}   \psi  =\psi  ,$ where $\psi,\psi\in L_{2}(\Omega).$
Hence, using  Theorem \ref{T1}, we have  that the  left and right side of \eqref{26} are finite. Therefore
using the Fubini  theorem, we have
$$
(\mathfrak{D}^{ \alpha }_{0+} f ,  g  )_{L_{2}(\Omega)}=\int\limits_{\omega}d\chi \int\limits_{0}^{d}\varphi(Q)
\overline{\left(\mathfrak{I } ^{\alpha}_{d-}\psi\right)(Q)}r^{n-1}dr 
$$
$$
 =\frac{1}{\Gamma(\alpha)}\int\limits_{\omega}d\chi\int\limits_{0}^{d}\varphi(Q)r^{n-1}dr\int\limits_{r}^{d}\frac{\overline{\psi(T)}}{(t-r)^{1-\alpha}}\,dt 
$$
\begin{equation}\label{27}
=\frac{1}{\Gamma(\alpha)}\int\limits_{\omega}d\chi\int\limits_{0}^{d}\overline{\psi(T)}t^{n-1}dt\int\limits_{0}^{t}\frac{\varphi(Q)}{(t-r)^{1-\alpha}}\left( \frac{r}{t}\right)^{n-1}dr 
$$
$$
=\int\limits_{\Omega} \left(\mathfrak{I} ^{\alpha}_{0+} \varphi\right)(Q)\,\overline{\psi(Q)} \, dQ=(f , \mathfrak{D }^{\alpha}_{d-} g )_{L_{2}(\Omega)}.
\end{equation}
Thus inequality \eqref{26} is proved. It follows that   $\mathrm{D}(\mathfrak{D }^{\alpha}_{d-})\subset\mathrm{D}\left( [\mathfrak{D}_{0+}^{ \alpha}]^{*}\right).$    Let us prove  that  $\mathrm{D}\left( [\mathfrak{D}_{0+}^{ \alpha}]^{*}\right)\subset
\mathrm{D}(\mathfrak{D }^{\alpha}_{d-}).$
In accordance with the definition of  adjoint  operator, we have
\begin{equation*}
\left(\mathfrak{D}^{ \alpha }_{0+} f,g \right)_{L_{2}(\Omega)}=\left(   f, [\mathfrak{D}_{0+}^{ \alpha} g]^{*}  \right)_{L_{2}(\Omega)},\,f\in \mathrm{D}(\mathfrak{D}^{ \alpha }_{0+}),\,g\in \mathrm{D}\left( [\mathfrak{D}_{0+}^{ \alpha}]^{*}\right).
\end{equation*}
 Note that since  $\mathrm{R}(\mathfrak{D }^{\alpha}_{d-})=L_{2}(\Omega),$ then  $\mathrm{R}
\left([\mathfrak{D}_{0+}^{ \alpha}]^{*}\right)=L_{2}(\Omega).$ Using the Fubini  theorem,  it can be  easily shown that
$$
\left( \mathfrak{D}^{ \alpha }_{0+} f,g- \mathfrak{I } ^{\alpha}_{d-}[\mathfrak{D}_{0+}^{ \alpha}]^{*}g    \right)_{L_{2}(\Omega)}=0.
$$
By virtue  of Theorem \ref{T3}, we have  $\mathrm{R}(\mathfrak{D}^{ \alpha }_{0+})=L_{2}(\Omega).$    Hence    $ g= \mathfrak{I } ^{\alpha}_{d-}[\mathfrak{D}_{0+}^{ \alpha}]^{*}g $ a.e.
It implies that $\mathrm{D}\left([\mathfrak{D}_{0+}^{ \alpha}]^{*}\right)\subset\mathrm{D}\left(\mathfrak{D }^{\alpha}_{d-}\right).$
\end{proof}

\section{ Strictly accretive  property}
 The  following theorem establishes the strictly accretive  property  (see \cite[p. 352]{firstab_lit:kato1966}) of the    Kipriyanov  fractional differential operator.\\\\

\begin{theorem}\label{T5}
  Suppose  $\rho(Q)$ is a real non-negative function,   $\rho\in{\rm Lip}\, \lambda,\; \lambda>\alpha;$
  then the following inequality    holds
\begin{equation}\label{28}
 {\rm Re} ( f,\mathfrak{D}^{\alpha}f)_{L_2(\Omega,\rho)}\geq\mu\|f\|^{2}_{L_2(\Omega,\rho)},\;f\in H^{1}_{0} (\Omega),
\end{equation}
where
$$
\mu=\frac{1}{2}\mathfrak{d}^{-\alpha} \left(  \Gamma^{-1}(1-\alpha) +C_{n}^{(\alpha)}\right)-
\frac{\alpha M  \mathfrak{d}^{\lambda-\alpha}  }{2\Gamma(1-\alpha)(\lambda-\alpha)\inf  \rho}\,.
$$
Moreover, if  we have in additional that for any fixed direction $\mathbf{e}$ the function  $\rho$ is   monotonically  non-increasing,   then
$$
\mu=\frac{1}{2}\mathfrak{d}^{-\alpha} \left(  \Gamma^{-1}(1-\alpha) +C_{n}^{(\alpha)}\right).
$$

\end{theorem}
\begin{proof}
Consider a real  case and let $f\in C_{0}^{\infty}(\Omega),$ we have
\begin{equation}\label{28.0}
\rho(Q)f(Q)( \mathfrak{D}^{\alpha}f )(Q)=\frac{1}{2}( \mathfrak{D}^{\alpha}\rho f^{2} )(Q)  
$$
$$
+\frac{\alpha}{2\Gamma(1-\alpha)} \int\limits_{0}^{r} \frac{\rho(Q)[f (Q)- f(T)]^{2}}{(r - t)^{\alpha+1}}
\left(\frac{t}{r}\right)^{n-1}\!\!\!dt 
$$
$$
+\frac{\alpha}{2\Gamma(1-\alpha)} \int\limits_{0}^{r} \frac{f^{2}(T)[\rho (T)- \rho(Q)] }{(r - t)^{\alpha+1}}
\left(\frac{t}{r}\right)^{n-1}\!\!\!dt+ \frac{C_{n}^{\alpha}}{2}(\rho f^{2})(Q)r^{-\alpha} 
$$
$$
=I_{0}(Q)+I_{1}(Q)+I_{2}(Q)+I_{3}(Q).
\end{equation}
Applying Theorem \ref{4}, we have
\begin{equation}\label{28.1}
\int\limits_{\Omega} I_{0}(Q)dQ=\frac{1}{2}\int\limits_{\Omega}(\mathfrak{D}^{\alpha}_{d-}1)(Q)   (\rho f^{2} )(Q)dQ  
$$
$$
=\frac{1}{2\Gamma(1-\alpha)}\int\limits_{\Omega}(d(\mathbf{e})-r)^{-\alpha}    (\rho f^{2} )(Q)dQ\geq \frac{\mathfrak{d}^{-\alpha}}{2\Gamma(1-\alpha)}\|f\|^{2}_{L_{2}(\Omega,\rho)}.
\end{equation}

Using the Fubini theorem, it can be shown in the usual way that
\begin{equation}\label{28.2}
\left|\int\limits_{\Omega} I_{2}(Q)dQ\right|\leq\frac{\alpha}{2\Gamma(1-\alpha)}\int\limits_{\omega}d \chi \int\limits_{0}^{d(\mathbf{e})}r^{n-1}dr \int\limits_{0}^{r} \frac{f^{2}(T)|\rho (T)- \rho(Q)| }{(r - t)^{\alpha+1}}
\left(\frac{t}{r}\right)^{n-1}\!\!\!dt 
$$
$$
=\frac{\alpha}{2\Gamma(1-\alpha)}\int\limits_{\omega}d \chi \int\limits_{0}^{d(\mathbf{e})}f^{2}(T) t^{n-1}dt \int\limits_{t}^{d(\mathbf{e})} \frac{|\rho (T)- \rho(Q)| }{(r - t)^{\alpha+1}}
 dr 
$$
$$
=\frac{\alpha}{2\Gamma(1-\alpha)}\int\limits_{\omega}d \chi \int\limits_{0}^{d(\mathbf{e})}f^{2}(T) t^{n-1}dt \int\limits_{0}^{d(\mathbf{e})-t}
\frac{|\rho (Q-\tau \mathbf{e})- \rho(Q)| }{\tau^{\alpha+1}}
 d\tau 
$$
$$
\leq\frac{\alpha M}{2\Gamma(1-\alpha)}\int\limits_{\omega}d \chi \int\limits_{0}^{d(\mathbf{e})}f^{2}(T) t^{n-1}dt \int\limits_{0}^{d(\mathbf{e})-t} \tau^{\lambda-\alpha+1}
 d\tau  
$$
$$
\leq\frac{\alpha M  \mathfrak{d}^{\lambda-\alpha}}{2\Gamma(1-\alpha)(\lambda-\alpha)}\|f\|^{2}_{L_{2}(\Omega)} .
\end{equation}
Consider
\begin{equation}\label{28.3}
\int\limits_{\Omega} I_{3}(Q)dQ=C^{(\alpha)}_{n}\int\limits_{\Omega}(\rho f^{2})(Q)r^{-\alpha}dQ\geq\frac{ C^{(\alpha)}_{n} \mathfrak{d}^{-\alpha}}{2}\|f\|^{2}_{L_{2}(\Omega,\rho)}.
\end{equation}
 Combining \eqref{28.0},\eqref{28.1},\eqref{28.2},\eqref{28.3}, and the fact that  $I_{1} $ is  non-negative, we obtain
\begin{equation}\label{28.4}
 ( f,\mathfrak{D}^{\alpha}f)_{L_2(\Omega,\rho)}\geq\mu\|f\|^{2}_{L_2(\Omega,\rho)},\;f\in C^{\infty}_{0} (\Omega).
\end{equation}
In the case when  for any fixed direction $\mathbf{e}$ the function  $\rho$ is   monotonically  non-increasing, we have    $I_{2}\geq 0.$ Hence \eqref{28.4} is   fulfilled. Now assume    that   $f\in H_{0}^{1} (\Omega),$ then  there exists   a sequence $\{f_k\}\in C^{\infty}_{0}(\Omega),\,
f_k\stackrel {H_{0}^1 }{\longrightarrow}f.
$
Using this fact,  it is not hard to prove that
$
f_k\stackrel {L_{2}(\Omega,\rho) }{\longrightarrow}f.
$
Using   inequality \eqref{3}, we prove  that
$
\|\mathfrak{D}^{\alpha} f \|_{L_{2}(\Omega,\rho)}   \leq C  \| f \|_{H_0 ^1(\Omega)}.
$
Therefore
$
 \mathfrak{D}^{\alpha}f_k\stackrel {L_{2}(\Omega,\rho) }{\longrightarrow} \mathfrak{D}^{\alpha}f.
$
Hence using the continuity property of the inner product, we get
$$
( f_k,\mathfrak{D}^{\alpha}f_k)_{L_{2}(\Omega,\rho)}\rightarrow ( f ,\mathfrak{D}^{\alpha}f )_{L_{2}(\Omega,\rho)} .
$$
Passing to the limit on the left and right side of inequality \eqref{28.4}, we obtain
\begin{equation}\label{28.5}
   ( f,\mathfrak{D}^{\alpha}f)_{L_2(\Omega,\rho)}\geq\mu\|f\|^{2}_{L_2(\Omega,\rho)},\;f\in H^{1}_{0} (\Omega).
\end{equation}
 Now let us  consider the  complex  case. Note that the following   equality is true
\begin{equation}\label{31}
{\rm Re} ( f,\mathfrak{D}^{\alpha}f)_{L_2(\Omega,\rho)}= ( u,\mathfrak{D}^{\alpha}u)_{L_2(\Omega,\rho)}+
( v,\mathfrak{D}^{\alpha}v)_{L_2(\Omega,\rho)},\;u={\rm Re}f,\,v={\rm Im}f.
\end{equation}

Combining   \eqref{31}, \eqref{28.5},  we obtain  \eqref{28}.
 \end{proof}
 \section{  Sectorial property }
Consider a uniformly elliptic operator with real   coefficients and the Kipriyanov fractional derivative    in the final term
\begin{equation*}
 Lu:=-  D_{j} ( a^{ij} D_{i}u)  +\rho\, \mathfrak{D}^{ \alpha }u,\;\; (i,j=\overline{1,n}) ,
 \end{equation*}
 $$
 \; \mathrm{D}(L)=H^{2}(\Omega)\cap H^{1}_{0}(\Omega),
 $$
 \begin{equation}\label{33}
 a^{ij}(Q)\in C^{1}(\bar{\Omega})  ,\,a^{ij}\xi _{i}  \xi _{j}  \geq a_{0} |\xi|^{2} ,\,a_{0}>0,
\end{equation}
 \begin{equation}\label{34}
\rho(Q)>0,\;\rho(Q)\in {\rm Lip\,\lambda},\,\alpha<\lambda\leq1.
\end{equation}
We assume in additional that $\mu>0,$ here we use  the denotation  that is used in Theorem \eqref{T5}.
We also   consider the formal adjoint  operator
\begin{equation*}
 L^{+}u:=-  D_{i} ( a^{ij} D_{j}u)  + \mathfrak{D}  ^{\alpha}_{ d -}\rho u  ,
 \end{equation*}
 $$
 \;\mathrm{D}(L^{+})=\mathrm{D}(L),
 $$
 and   the   operator
\begin{equation*}
  H=\frac{1}{2}(L+L^{+}).
 \end{equation*}
 We   use  a special case of the Green  formula
 \begin{equation}\label{35}
-\int\limits_{\Omega}D_{j}(a^{ij}D_{i}u)\,\bar{v}\, dQ=\int\limits_{\Omega}a^{ij}D_{i}u\, \overline{D_{j}v}\,  dQ\,,\;u\in H^{2}(\Omega),v\in H_{0}^{1}(\Omega) .
\end{equation}
\begin{remark}\label{L3}
  The operators $L,L^{\!+}\!,H$  are closeable. We can easily check this fact,  if we apply    Theorem 3.4 \cite[p.337]{firstab_lit:kato1966}.
\end{remark}
We have the following lemma.
\begin{theorem}
\label{T6}
 The operators $\tilde{L},\,\tilde{L}^{+}$  are  strictly  accretive, their  numerical range  belongs to  the sector
\begin{equation*}
   \mathfrak{S}:= \{\zeta\in\mathbb{C}: \,|{\rm arg }\,(\zeta-\gamma)|\leq\theta\},
\end{equation*}
 where $\theta$ and $\gamma$ are defined by the coefficients of the operator $L.$
\end{theorem}
\begin{proof}
Consider the operator $L .$ It is not hard to prove that
\begin{equation}\label{37}
-{\rm Re} \left( D_{j} [a^{ij} D_{i}f] ,f\right)_{L_{2}(\Omega)}\geq a_{0}\|   f  \|^{2}_{L^{1}_{2}(\Omega)},\;f\in \mathrm{D}(L).
\end{equation}
Hence
\begin{equation}\label{40}
  {\rm Re}  (f_{n}, L f_{n}  )_{L_{2}(\Omega)}\geq a_{0}\|   f_{n} \|^{2}_{L^{1}_{2}(\Omega)}+ {\rm Re}(f_{n},\mathfrak{D}^{\alpha}f_{n})_{L_2(\Omega,\rho)} ,\;\{f_{n}\}\subset\mathrm{D}(L).
\end{equation}
Assume that  $ f\in \mathrm{D}(\tilde{L}).$ In accordance with the definition, there  exists a sequence $\{f_{n}\}\subset\mathrm{D}(L),\,f_{n}\xrightarrow[  L ]{}f.$  By virtue of  \eqref{40},  we easily  prove that  $f\in H_{0}^{1}(\Omega).$  Using the continuity property of the inner product, we pass to the limit on the left and right side of inequality \eqref{40}. Thus, we have
\begin{equation}\label{41}
 {\rm Re}( f , \tilde{L} f   )_{L_{2}(\Omega)}\geq a_{0}\|   f  \|^{2}_{L^{1}_{2}(\Omega)}+ {\rm Re}( f ,\mathfrak{D}^{\alpha}f )_{L_2(\Omega,\rho)} ,\; f  \in\mathrm{D}(\tilde{L}).
\end{equation}
 By virtue of Theorem \ref{T5}, we can rewrite the previous inequality as follows
 \begin{equation}\label{42}
  {\rm Re}( f,\tilde{L}f  )_{L_{2}(\Omega)}\geq a_{0}\|f\|^{2}_{L^{1}_{2}(\Omega)}+\mu\|f\|^{2}_{L_{2}(\Omega,\rho)} ,\;f\in \mathrm{D}(\tilde{L}).
\end{equation}
Applying the     Friedrichs inequality    to the first summand of   the right side, we  get
  \begin{equation}\label{43}
  {\rm Re}  ( f,\tilde{L}f )_{L_{2}(\Omega)}\geq \mu_{1}\|f\|^{2}_{L_{2}(\Omega)},\;
 f  \in\mathrm{D}(\tilde{L}),\; \mu_{1}=a_{0}+\mu\inf\rho(Q)  .
 \end{equation}
Consider the  imaginary component   of the form, generated by the operator $L$  
\begin{equation}\label{44}
\left|{\rm Im} ( f,Lf    )_{L_{2}(\Omega)}\right|\leq   \left|\int\limits_{\Omega}
\left(a^{ij}  D_{i}u D_{j}v-a^{ij}  D_{i}v D_{j}u\right)dQ\right| 
$$
$$
+\left|  ( u,\mathfrak{D}^{\alpha}v )_{L_{2}(\Omega,\rho)} -( v,\mathfrak{D}^{\alpha}u   )_{L_{2}(\Omega,\rho)}\right|= I_{1}+I_{2}.
 \end{equation}
 Using the Cauchy Schwarz inequality for sums, the Young inequality,  we have
\begin{equation}\label{45}
a^{ij}  D_{i}u D_{j}v\leq a  |Du| |Dv| \leq
 \frac{a}{2}  \left(
|Du|^{2} +|Dv|^{2} \right),a(Q)\!=\!\left(\sum\limits_{i,j=1}^{n}
|a_{ij}(Q)|^{2} \right)^{\!\!1/2}\!\!.
\end{equation}
Hence
\begin{equation}\label{45.1}
 I_{1}\leq   a_{1}\|f\|^{2}_{L^{1}_{2}(\Omega)},\;a_{1}=\sup  a(Q)  .
\end{equation}
 Applying    inequality \eqref{3}, the Young inequality,    we get
\begin{equation}\label{46}
\left|( u,\mathfrak{D}^{\alpha}v )_{L_{2}(\Omega,p)}\right|\leq C_{1}\|u\|_{L_{2}(\Omega)}\|\mathfrak{D}^{\alpha}v\|_{L_{q}(\Omega)} 
$$
$$
 \leq
C_{1} \|u\|_{L_{2}(\Omega)}\left\{\frac{K}{\delta^{\nu}}\|v\|_{L_{2}(\Omega)}+\delta^{1-\nu}\|v\|_{L^{1}_{2}(\Omega)} \right\}  
$$
$$
\leq\frac{1}{\varepsilon} \|u\|^{2}_{L_{2}(\Omega)} + \varepsilon\left(\frac{ KC_{1}}{\sqrt{2}\delta^{\nu}}\right)^{2} \|v\|^{2}_{L_{2}(\Omega)}  +
 \frac{\varepsilon}{2}\left( C_{1}\delta^{1-\nu}\right)^{2} \|v\|^{2}_{L^{1}_{2}(\Omega)},
\end{equation}
where $ 2<q< 2n/(2\alpha-2+n).$
Hence
$$
 I_2   \leq\left|( u,\mathfrak{D}^{\alpha}v )_{L_{2}(\Omega,\rho)}\right|+\left|( v,\mathfrak{D}^{\alpha}u )_{L_{2}(\Omega,\rho)}\right|\leq \frac{1}{\varepsilon}\left(\|u\|^{2}_{L_{2}(\Omega)}+\|v\|^{2}_{L_{2}(\Omega)}\right) 
 $$
 $$
 +\varepsilon\left(\frac{ KC_{1}}{\sqrt{2}\delta^{\nu}}\right)^{2}\left(\|u\|^{2}_{L_{2}(\Omega)}+\|v\|^{2}_{L_{2}(\Omega)} \right)+\frac{\varepsilon}{2}\left( C_{1} \delta^{1-\nu}\right)^{2}\left(\|u\|^{2}_{L^{1}_{2}(\Omega)}+\|v\|^{2}_{L^{1}_{2}(\Omega)} \right) 
 $$
\begin{equation}\label{48}
= \left(\varepsilon \delta^{-2\nu}C_{2}  +\frac{1}{\varepsilon}\right) \|f\|^{2}_{L_{2}(\Omega)} +
 \varepsilon \delta^{2-2\nu} C_{3} \|f\|^{2}_{L^{1}_{2}(\Omega)}  .
\end{equation}
Taking into account  \eqref{44} and combining  \eqref{45.1},  \eqref{48}, we easily prove that
$$
\left|{\rm Im} ( f,\tilde{L}f    )_{L_{2}(\Omega)}\right|\leq
 \left(\varepsilon \delta^{-2\nu}C_{2}  +\frac{1}{\varepsilon}\right)   \|f\|^{2}_{L_{2}(\Omega)} + \left(\varepsilon \delta^{2-2\nu} C_{3}+a_{1}\right)   \|  f\|^{2}_{L_{2}^{1}(\Omega)}, f\in \mathrm{D}(\tilde{L}).
 $$
Thus  by virtue of    \eqref{43}   for an arbitrary number  $k>0,$  the next inequality holds
$$
{\rm Re}( f,\tilde{L}f    )_{L_{2}(\Omega)}-k \left|{\rm Im} ( f,\tilde{L}f    )_{L_{2}(\Omega)}\right|\geq
 \left(a_{0}-k    [\varepsilon\delta^{2-2\nu} C_{3}+a_{1}] \right)\|  f\|^{2}_{L_{2}^{1}(\Omega)} 
 $$
 $$
 +\left( \mu \inf \rho(Q)- k\left[\varepsilon \delta^{-2\nu}C_{2}  +\frac{1}{\varepsilon}\right] \right)\|f\|^{2}_{L_{2}(\Omega)}.
$$
 Choose  $k= a_{0}\left(\varepsilon \delta^{2-2\nu} C_{3}+ a_{1} \right)^{-1}\!\!,$ we get
\begin{equation}\label{49}
\left|{\rm Im} ( f,(\tilde{L}-\gamma) f    )_{L_{2}(\Omega)}\right|   \leq \frac{1}{k}{\rm Re}( f,(\tilde{L}-\gamma)f    )_{L_{2}(\Omega)},
$$
$$
\;\gamma= \mu \inf \rho(Q)- k\left[\varepsilon \delta^{-2\nu}C_{2}  +\frac{1}{\varepsilon}\right].
\end{equation}
This  inequality shows  that the numerical range    $\Theta(\tilde{L})$ belongs to the sector with the top   $\gamma$ and the semi-angle $\theta=\arctan(1/k).$
The prove corresponding  to the operator $\tilde{L}^{+}$ is analogous.
 \end{proof}
We do not   study in detail   the conditions under which  $\gamma>0,$  but we   just   note that relation     \eqref{49} gives us an opportunity to formulate them in an easy way. Further,     we assume that the coefficients of the operator $L$ such that $\gamma>0.$

 \begin{theorem}\label{T7}
   The operators $\tilde{L},\tilde{L}^{+},\tilde{H}$   is m-sectorial, the  operator $\tilde{H}$ is selfadjoint.
\end{theorem}
\begin{proof} By virtue  of  Theorem \ref{T6} we have  that the operator $\tilde{L}$ is sectorial i.e.
 $\Theta(L)\subset  \mathfrak{S}.$  Applying Theorem 3.2   \cite[p. 336]{firstab_lit:kato1966}
 we   conclude that $\mathrm{R}(\tilde{L}-\zeta)$ is a closed  space  for any $\zeta\in \mathbb{C}\setminus\mathfrak{S}$ and that the next relation     holds
\begin{equation}\label{50}
  {\rm def}(\tilde{L}-\zeta)=\eta,\; \eta={\rm const} .
 \end{equation} Using   \eqref{43}, it is not hard to prove that   $
  \|\tilde{L}f\|_{L_{2}(\Omega)}\geq \sqrt{\mu_{1}}\|f\|_{L_{2}(\Omega)},\,
 f  \in\mathrm{D}(\tilde{L}).
$ Hence the inverse operator $(\tilde{L}+\zeta)^{-1}$ is defined  on the subspace $\mathrm{R}(\tilde{L}+\zeta),\,{\rm Re}\zeta>0.$     In accordance with    condition (3.38) \cite[p.350]{firstab_lit:kato1966},   we need to show  that
\begin{equation}\label{51}
{\rm def}(\tilde{L}+\zeta)=0,\;\|(\tilde{L}+\zeta)^{-1}\|\leq ({\rm Re}\zeta)^{-1},\;{\rm Re}\zeta>0 .
\end{equation}
Since $\gamma>0,$  then  the left half-plane is  included    in the the set $\mathbb{C}\setminus \mathfrak{S}.$
Note that by virtue of inequality  \eqref{43}, we have
 \begin{equation}\label{52}
  {\rm Re}  ( f,(\tilde{L}-\zeta )f  )_{L_{2}(\Omega)}\geq  (\mu- {\rm Re} \zeta ) \|f\|^{2}_{L_{2}(\Omega)} .
 \end{equation}
 Let $\zeta_{0}\in\mathbb{C}\setminus \mathfrak{S},\;{\rm Re}\zeta_{0} <0.$
  Since the  operator $\tilde{L}-\zeta_{0}$ has a closed range     $\mathrm{R} (\tilde{L}-\zeta_{0}),$ then we have
\begin{equation*}
 L_{2}(\Omega)=\mathrm{R} (\tilde{L}-\zeta_{0})\oplus \mathrm{R} (\tilde{L}-\zeta_{0})^{\perp} .
 \end{equation*}
Note that   $  C^{\infty}_{0}(\Omega)\cap \mathrm{R} (\tilde{L}-\zeta_{0})^{\perp}=0,$   because if we assume   the contrary, then applying inequality  \eqref{52}  for any element
 $u\in C^{\infty}_{0}(\Omega)\cap \mathrm{R}  (\tilde{L}-\zeta_{0})^{\perp},$    we get
 \begin{equation*}
(\mu-{\rm Re}\zeta_{0}) \|u\|^{2}_{L_{2}(\Omega)} \leq {\rm Re} ( u,(\tilde{L}-\zeta_{0})u  )_{L_{2}(\Omega)}=0,
 \end{equation*}
hence $u=0.$ Thus  this fact   implies that
$$
\left(g,v\right)_{L_{2}(\Omega)}=0,\,g\in  \mathrm{R}  (\tilde{L}-\zeta_{0})^{\perp},\, \in C^{\infty}_{0}(\Omega).
$$
Since $  C^{\infty}_{0}(\Omega)$  is a dense set in $L_{2}(\Omega),$ then $\mathrm{R}  (\tilde{L}-\zeta_{0})^{\perp}=0.$ It follows that  ${\rm def} (\tilde{L}-\zeta_{0}) =0.$ Now if we note \eqref{50}
then we  came to the conclusion that ${\rm def} (\tilde{L}-\zeta )=0,\;\zeta\in \mathbb{C}\setminus\mathfrak{S}.$  Hence ${\rm def} (\tilde{L}+\zeta )=0,\;{\rm Re}\zeta>0.$   Thus the proof  of  the first relation of  \eqref{51} is complete.
To prove the second relation \eqref{51} we should  note that
 \begin{equation*}
(\mu +{\rm Re}\zeta ) \|f\|^{2}_{L_{2}(\Omega)} \leq {\rm Re} ( f,(\tilde{L}+\zeta )f  )_{L_{2}(\Omega)}\leq \|f\|_{L_{2}(\Omega)}\|(\tilde{L}+\zeta )\|_{L_{2}(\Omega)},
 \end{equation*}
 $$
 \;f\in \mathrm{D}(\tilde{L}),\;
{\rm Re}\zeta>0 .
 $$
Using    first relation   \eqref{51}, we have
 \begin{equation*}
\|(\tilde{L}+\zeta )^{-1}g\|_{L_{2}(\Omega)}     \leq(\mu  +{\rm Re}\,\zeta ) ^{-1} \|g\|_{L_{2}(\Omega)}\leq ( {\rm Re}\,\zeta ) ^{-1} \|g\|_{L_{2}(\Omega)},\,g\in L_{2}(\Omega).
 \end{equation*}
 This implies that
$$
\|(\tilde{L}+\zeta )^{-1} \| \leq( {\rm Re}\,\zeta ) ^{-1},\,{\rm Re}\zeta>0.
$$
This concludes the proof  corresponding to the operator  $\tilde{L}.$  The proof  corresponding to the operator  $\tilde{L}^{+}$ is     analogous.
Consider the  operator $\tilde{H}.$ It is obvious that   $ \tilde{H}  $ is a symmetric operator. Hence   $ \Theta(\tilde{H})\subset \mathbb{R}.  $ Using \eqref{40} and arguing as above, we see that
 \begin{equation}\label{52.0}
    ( f,\tilde{H}f  )_{L_{2}(\Omega)}\geq  \mu_{1}  \|f\|^{2}_{L_{2}(\Omega)} .
 \end{equation}
 Continuing the used above  line of reasoning  and applying Theorem 3.2 \cite[p.336]{firstab_lit:kato1966},  we see that
  \begin{equation}\label{52.1}
{\rm def}(\tilde{H}-\zeta )=0,\,{\rm Im }\zeta\neq 0;
\end{equation}
  \begin{equation}\label{52.2}
   {\rm def}(\tilde{H}+\zeta)=0,\;\|(\tilde{H}+\zeta)^{-1}\|\leq ({\rm Re}\zeta)^{-1},\;{\rm Re}\zeta>0 .
\end{equation}
Combining \eqref{52.1} with   Theorem 3.16  \cite[p.340]{firstab_lit:kato1966}, we conclude that  the operator $\tilde{H}$ is selfadjoint.
Finally, note that in accordance with the definition,   relation  \eqref{52.2} implies   that the  operator $\tilde{H}$ is m-accretive.
  Since we  already know  that the operators $\tilde{L},\tilde{L}^{+},\tilde{H}$ are sectorial and m-accretive, then  in accordance with the definition they are m-sectorial.
 \end{proof}

\section{Main theorems}

In this section we need using the theory of sesquilinear forms. If it is not stated otherwise, we use the definitions and  the  notation  of the monograph \cite {firstab_lit:kato1966}.
Consider the forms
\begin{equation}\label{53}
t[u,v] = \int\limits_{\Omega} a^{ij}D_{i}u\, \overline{D_{j}v}dQ+ \int\limits_{\Omega}  \rho\,\mathfrak{D}^{\alpha }u \, \bar{v } dQ, \; u,v\in H^{1}_{0}(\Omega) ,
\end{equation}
  \begin{equation*}
t^{*}[u,v]: =\overline{t [v,u]} =\int\limits_{\Omega} a^{ij}D_{j}u\, \overline{D_{i}v}dQ+ \int\limits_{\Omega}   u  \rho\,  \overline{\mathfrak{D}^{\alpha }v }dQ,
\end{equation*}
$$
\mathfrak{Re} t :=\frac{1}{2}(t+t^{*}).
$$
For convenience, we use the shorthand notation $h:=\mathfrak{Re} t.$

 \begin{lemma}\label{L4}
The form $t$ is a closed sectorial form, moreover  $t=\mathfrak{\tilde{f}}, $ where
 \begin{equation*}
 \mathfrak{ f }[u,v]=(\tilde{L}u,v)_{L_{2}(\Omega)},\;u,v\in \mathrm{D}(\tilde{L}).
 \end{equation*}
\end{lemma}
 \begin{proof}
First we shall show  that the following    inequality   holds
\begin{equation}\label{54}
C_{0}\|f\|^{2} _{H^{1}_{0}(\Omega)}\leq  \left|t[f ]\right|\leq C_{1}\|f\|^{2} _{H^{1}_{0}(\Omega)},\,f\in H^{1}_{0}(\Omega).
\end{equation}
Using \eqref{41}, Theorem \ref{T5},   we obtain
\begin{equation}\label{55}
C_{0}\|f\|^{2} _{H^{1}_{0}(\Omega)}\leq   {\rm Re} t[f ] \leq\left|t[f ]\right|,\;f\in H^{1}_{0} (\Omega).
\end{equation}
Applying \eqref{45},\eqref{46}, we get
\begin{equation}\label{56}
|t[f]|\leq\left|\left(a^{ij}D_if,D_jf\right)_{\!L_{2}(\Omega)}\!\right|+\left|\left(\rho \,\mathfrak{D}^{\alpha} f, f\right)_{\!L_{2}(\Omega)}\!\right| \leq C_{1}\|f\|^{2}_{H^{1}_{0}(\Omega)},\,f\in H^{1}_{0} (\Omega).
\end{equation}
Note  that $ H^{1}_{0}(\Omega) \subset \mathrm{D}( \tilde{t}) .$  If $f\in \mathrm{D}( \tilde{t} ),$ then in accordance with the definition,  there exists  a sequence
$
\{f_{n}\}\subset \mathrm{D}( t),\, f_{n}\xrightarrow[t ]{ }f.
$
Applying   \eqref{54}, we get  $ f_{n}\xrightarrow[  ]{ H^{1}_{0}}f .$
Since the space $H^{1}_{0}(\Omega)$ is complete, then   $\mathrm{D}( \tilde{t}) \subset H^{1}_{0}(\Omega).$ It implies that $\mathrm{D}( \tilde{t})=\mathrm{D}( t).$ Hence   $t$ is a closed form.
The proof of the sectorial property contains in the proof of Theorem \ref{T6}. Let us prove that $t=\mathfrak{\tilde{f}}.$  First, we shall    show that
\begin{equation}\label{57}
\mathfrak{ \mathfrak{f} }[u,v]=t[u,v],\;u,v\in \mathrm{D}(\mathfrak{f}).
\end{equation}
 Using   formula \eqref{35}, we have
\begin{equation}\label{58}
( L u ,v )_{L_{2}(\Omega)}=t[u ,v ],\;u,v\in \mathrm{D}(L).
 \end{equation}
 Hence  we can rewrite    relation \eqref{54} in the following form
\begin{equation}\label{59}
C_{0}\|f\|^{2} _{H^{1}_{0}(\Omega)}\leq  \left|( L f,f)_{L_{2}(\Omega)}\right|\leq C_{1}\|f\|^{2} _{H^{1}_{0}(\Omega)},\;f\in \mathrm{D}(L).
\end{equation}
 Assume that $f \in \mathrm{D}(\tilde{L}),$ then there exists a sequence $\{f_{n}\}\in \mathrm{D}( L ),\,f_{n}\xrightarrow[L]{}f.$   Combining   \eqref{59},\eqref{54}, we obtain
   $f_{n}\xrightarrow[t]{}f.$   These facts give us an opportunity to pass to the limit on the left and right side of \eqref{58}. Thus, we obtain \eqref{57}. Combining  \eqref{57},\eqref{54}, we get
\begin{equation}\label{60}
C_{0}\|f\|^{2} _{H^{1}_{0}(\Omega)}\leq  \left|\mathfrak{ f }[f ]\right|\leq C_{1}\|f\|^{2} _{H^{1}_{0}(\Omega)},\;f\in \mathrm{D}(\mathfrak{f}).
\end{equation}
Note that by virtue  of Theorem \ref{T6} the operator $\tilde{L}$ is sectorial, hence due to Theorem 1.27 \cite[p.399] {firstab_lit:kato1966} the form $\mathfrak{f}$ is closable. Using the facts established above, Theorem 1.17 \cite[p.395] {firstab_lit:kato1966},   passing to the limit  on the left and right side of inequality \eqref{57}, we get
\begin{equation*}
\mathfrak{ \mathfrak{\tilde{f}} }[u,v]=t[u,v],\;u,v\in H_{0}^{1}(\Omega).
\end{equation*}
This concludes the proof.
 \end{proof}

\begin{lemma}\label{L5}
The form h is a closed symmetric sectorial form, moreover  $h=\mathfrak{\tilde{k}}, $ where
 \begin{equation*}
 \mathfrak{ k }[u,v]=(\tilde{H}u,v)_{L_{2}(\Omega)},\;u,v\in \mathrm{D}(\tilde{H}).
 \end{equation*}
\end{lemma}
\begin{proof}
To prove  the  symmetric  property   (see(1.5) \cite[p.387] {firstab_lit:kato1966}) of the form $h,$  it is  sufficient to note that
$$
h[u,v]=\frac{1}{2}\left(t[u,v]+ \overline{t[v,u]}  \right)=\frac{1}{2}\overline{\left( t[v,u] +\overline{t[u,v]}\right)}=\overline{h[v,u]},\;u,v\in \mathrm{D}(h).
$$
Obviously, we have
$
h[f]= {\rm Re}\,t[f].
$
Hence applying    \eqref{55},   \eqref{56},   we have
\begin{equation}\label{61}
C_{0}\|f\| _{H^{1}_{0}(\Omega)}\leq  h[f ] \leq C_{1}\|f\| _{H^{1}_{0}(\Omega)},\;f\in H^{1}_{0}(\Omega).
\end{equation}
 Arguing as above, using \eqref{61}, it is easy to  prove   that $\mathrm{D}(\tilde{h})=H_{0}^{1}(\Omega). $ Hence the form $h$ is a closed form. The  proof of the sectorial property of the form $h$ contains in the proof of Theorem \ref{T6}.
Let us  prove that $h=\mathfrak{\tilde{k}}.$ We shall  show that
\begin{equation}\label{62}
\mathfrak{ \mathfrak{k} }[u,v]=h[u,v],\,u,v\in \mathrm{D}(\mathfrak{k}).
\end{equation}
Applying \ref{L1}, Lemma \ref{L2}, we have
$$
(\rho\,\mathfrak{D}^{\alpha}f,g)_{L_{2}(\Omega)}=(f,\mathfrak{D}_{d-}^{\alpha}\rho g)_{L_{2}(\Omega)},\,f,g\in H_{0}^{1}(\Omega).
$$
Combining this fact with formula \eqref{35}, it is not hard to prove that
\begin{equation}\label{63}
( H u,v)_{L_{2}(\Omega)}=h[u,v],\;u,v\in \mathrm{D}(H).
\end{equation}
Using \eqref{63}, we can rewrite  estimate \eqref{61}  as  follows
\begin{equation}\label{64}
C_{0}\|f\| _{H^{1}_{0}(\Omega)}\leq    ( H f,f)_{L_{2}(\Omega)} \leq C_{1}\|f\| _{H^{1}_{0}(\Omega)},\;f\in \mathrm{D}(H).
\end{equation}
Note that in consequence of  Remark  \ref{L3} the  operator $H$ is closeable. Assume that  $f \in \mathrm{D}(\tilde{H}),$    then there  exists a sequence  $\{f_{n}\}\subset \mathrm{D}(H),\,f_{n}\xrightarrow[H]{}f .$   Combining  \eqref{64},\eqref{61}, we obtain
 $f_{n}\xrightarrow[h]{}f.$    Passing to the limit on the left and right side of \eqref{63}, we get \eqref{62}.
 Combining  \eqref{62},\eqref{61}, we obtain
\begin{equation}\label{65}
C_{0}\|f\| _{H^{1}_{0}(\Omega)}\leq   \mathfrak{ k }[f ] \leq C_{1}\|f\| _{H^{1}_{0}(\Omega)},\;f\in \mathrm{D}(\mathfrak{k}).
\end{equation}
 Note that in consequence of Theorem \ref{T6} the operator $\tilde{H}$ is sectorial. Hence by virtue of    Theorem 1.27 \cite[p.399] {firstab_lit:kato1966} the form $\mathfrak{k}$ is closable.
Using the proven  above facts, Theorem 1.17 \cite[p.395] {firstab_lit:kato1966},  passing to the limits on the left and right side of inequality \eqref{62}, we get
\begin{equation*}
 \mathfrak{ \mathfrak{\tilde{k}} }[u,v]  = h[u,v],\,u,v\in H_{0}^{1}(\Omega).
\end{equation*}
This completes the proof.
\end{proof}

\begin{theorem}\label{T8}
The operator $\tilde{H}$ has a compact resolvent,    the following estimate   holds
\begin{equation}\label{66}
\lambda_{n}(L_{0})\leq\lambda_{n}(\tilde{H})\leq\lambda_{n}(L_{1}),\,n\in\mathbb{N},
\end{equation}
where $\lambda_{n}(L_{k}),\,k=0,1$ are respectively  the eigenvalues of the following   operators with real  constant coefficients 
\begin{equation}\label{67}
L_{k}f=-  a_{ k }^{ij} D _{j}D_{i}f +\rho_{k}f,\,\mathrm{D}(L_{k})=\mathrm{D}(L),
$$
$$
a_{ k }^{ij}\xi_{i}\xi_{j}>0,\,\rho_{k}>0.
\end{equation}
\end{theorem}
\begin{proof}
  First, we shall prove  the following      propositions \\
i) {\it The  operators $\tilde{H},L_{k}$ are positive-definite.} Using the fact that   the operator $H$ is selfadjoint, relation \eqref{52.0}, we conclude that
the  operator  $\tilde{H}$ is positive-definite. Using the definition, we can easily prove    that the operators $L_{k}$ are   positive-definite.   \\
ii) {\it The space $H^{1}_{0}(\Omega)$ coincides  with the energetic spaces  $\mathfrak{H}_{\tilde{H}},\mathfrak{H}_{L_{k}}$  as a set of   elements.} Using   Lemma \ref{L5}, we have
  \begin{equation}\label{68}
 \|f\| _{\mathfrak{H}_{\tilde{H}}}=\tilde{\mathfrak{k}}[f]=h[f],\;f\in H_{0}^{1}(\Omega).
\end{equation}
Hence the space  $\mathfrak{H}_{\tilde{H}}$ coincides with $H^{1}_{0}(\Omega)$ as a set of   elements. Using this fact,
we obtain   the coincidence of the spaces $H^{1}_{0}(\Omega)$ and $\mathfrak{H}_{L_{k}}$ as the particular case. \\
iii){\it We have the following estimates
\begin{equation}\label{69}
 \|f\| _{\mathfrak{H} _{L_{0}}}\leq \|f\| _{\mathfrak{H}_{\tilde{H}}}\leq  \|f\| _{\mathfrak{H} _{L_{1}}},\,f\in H^{1}_{0}(\Omega).
\end{equation}
}
We obtain the equivalence  of the   norms $\|\cdot\|_{H_{0}^{1}}$ and $\|\cdot\|_{\mathfrak{H} _{L_{k}}}$ as the particular case of relation \eqref{54}.
It is obvious that there   exist  such   operators  $L_{k}$    that the next  inequalities hold
\begin{equation}\label{70}
\|f\| _{\mathfrak{H} _{L_{0}}}\leq C_{0}\|f\| _{H^{1}_{0}(\Omega)},\;C_{1}\|f\| _{H^{1}_{0}(\Omega)}\leq\|f\| _{\mathfrak{H} _{L_{1}}},\,f\in H^{1}_{0}(\Omega).
\end{equation}
Combining \eqref{61},\eqref{68},\eqref{70}, we get   \eqref{69}.

 Now we can prove the   proposal  of this theorem.  Note that   the operators $\tilde{H},$ $L_{k}$  are positive-definite,    the   norms 
 $\|\cdot\|_{H_{0}^{1}}, \|\cdot\|_{\mathfrak{H} _{L_{k}}}, \|\cdot\|_{\mathfrak{H}_{\tilde{H}}}$     are equivalent.   Applying  the  Rellich-Kondrashov  theorem, we have that the energetic  spaces
  $\mathfrak{H}_{\tilde{H}},\;\mathfrak{H} _{L_{k}}$  are compactly embedded into $L_{2}(\Omega).$ Using Theorem 3 \cite[p.216]{firstab_lit:mihlin1970}, we obtain  the fact  that  the operators $ L_{0} ,L_{1},\tilde{H}$ have  a discrete   spectrum.
 Taking into account   (i),(ii),(iii), in accordance with the definition \cite[p.225]{firstab_lit:mihlin1970}, we have
$$
L_{0}\leq \tilde{H} \leq L_{1}.
$$
 Applying     Theorem 1 \cite[p.225]{firstab_lit:mihlin1970},   we obtain  \eqref{66}.
 Note that by virtue of Theorem \ref{T7} the operator $\tilde{H}$ is m-accretive. Hence $0\in P(\tilde{H}).$ Due to    Theorem 5 \cite[p.222]{firstab_lit:mihlin1970} the operator $\tilde{H}$ has a compact resolvent at the point zero.
 Applying    Theorem 6.29 \cite[p.237]{firstab_lit:kato1966}, we conclude that the  operator $\tilde{H}$ has a compact resolvent.

\end{proof}

\begin{theorem}
Operator $\tilde{L}$ has a compact resolvent,   discrete spectrum.
\end{theorem}
\begin{proof}
Note that in accordance with Theorem \ref{T7} the operators $\tilde{L},\tilde{H}$ are m-sectorial, the operator $\tilde{H}$  is self-adjoint. Applying Lemma \ref{L4}, Lemma \ref{L5}, Theorem 2.9 \cite[p.409]{firstab_lit:kato1966}, we get   $T_{t}=\tilde{L},\;T_{h}=\tilde{H},$ where $T_{t},T_{h}$ are the Fridrichs  extensions    of the operators $\tilde{L}, \tilde{H}$ (see \cite[p.409]{firstab_lit:kato1966}) respectively. Since in accordance
with  the   definition \cite[p.424]{firstab_lit:kato1966} the operator $\tilde{H}$ is a real part of the operator $\tilde{L},$ then due to  Theorem
\ref{T8}, Theorem 3.3 \cite[p.424]{firstab_lit:kato1966} the operator $\tilde{L}$ has a compact resolvent. Applying  Theorem 6.29 \cite[p.237]{firstab_lit:kato1966}, we conclude that the operator $\tilde{L}$ has a     discrete   spectrum.
\end{proof}
\begin{remark}
 It can easily be checked   that the Kypriaynov operator   is reduced  to the Marchaud operator in the one-dimensional case. At the same time,
 the results of this work are only true   for  the  dimensions   $ n\geq2.$
 However,   using   Corollary 1 \cite{firstab_lit:1kukushkin2017}, which establishes the strictly  accretive property  of the Marchaud  operator, we can apply the obtained  technique to  the one-dimensional case.
\end{remark}

\section{Conclusions}

  The paper  presents  the results  obtained in   the spectral theory of  fractional differential  operators. A  number of   propositions of  independent interest in the fractional calculus theory  are  proved, the  new concept   of a multidimensional directional fractional integral    is introduced.   The sufficient conditions of the representability   by the directional  fractional integral are formulated.  In particular,   the inclusion of  the Sobolev space to  the class  of   functions  that are  representable  by the directional  fractional integral  is established. Note that the technique of the  proofs, which is analogous to  the one-dimensional case, is of particular interest. It should be noted that the  extension of the Kipriyanov fractional differential  operator  is obtained, the  adjoint  operator is found, and the strictly accretive property is proved. These all create   a complete description  reflecting   qualitative properties of   fractional differential operators.    As the main   results, the following  theorems establishing the  properties of an uniformly elliptic operator with  the  Kipriyanov fractional derivative in the final term   are  proved:  the  theorem on the  strictly  accretive  property,  the theorem on the  sectorial  property,  the theorem on the m-accretive  property, the  theorem  establishing   a two-sided estimate for the  eigenvalues  and    discreteness of the spectrum of the  real component. Using  the  sesquilinear forms theory,  we  obtained the  major  theoretical results. We consider the proofs corresponding to the multidimensional case, however the reduction to the one-dimensional case is possible. For instance, the one-dimensional case is described in the paper  \cite {firstab_lit:2kukushkin2017}. We also note  that the results in this direction can be obtained  for the real axis. It is worth noticing  that the application of the  sesquilinear forms theory,  as a tool to study second order differential operators with a fractional derivative in the final term,  gives  an opportunity to analyze  the major role  of the senior term in the functional properties of the operator. This technique is novel  and can be used for  studying  the spectrum of  perturbed  fractional differential operators. Therefore,
  the idea of the proof may be of interest regardless of the results.

\subsection*{Acknowledgments}
  The author  thanks Professor Alexander L. Skubachevskii for valuable remarks and comments made during the  report,  which took place 31.10.2017 at Peoples' Friendship University of Russia, Moscow.

\newpage

\end{document}